\documentclass[a4paper,twoside,english]{scrartcl}
\usepackage[T1]{fontenc}
\usepackage[latin1]{inputenc}
\usepackage{verbatim}
\usepackage{amsthm}
\usepackage{amsmath}
\usepackage{amssymb}

\makeatletter

\theoremstyle{plain}
\newtheorem{thm}{\protect\theoremname}[section]
  \theoremstyle{definition}
  \newtheorem{defn}[thm]{\protect\definitionname}
  \theoremstyle{plain}
  \newtheorem{prop}[thm]{\protect\propositionname}
  \theoremstyle{remark}
  \newtheorem{rem}[thm]{\protect\remarkname}
  \theoremstyle{definition}
  \newtheorem{example}[thm]{\protect\examplename}

\usepackage[T1]{fontenc}    

\usepackage{a4wide}        
\usepackage{amsmath}
\usepackage{url}

\usepackage{euscript}
\usepackage{mathtools}
\allowdisplaybreaks[1] 
\usepackage{amsfonts}
\newenvironment{keywords}{ \noindent\footnotesize\textbf{Keywords and phrases:}}{}

\newenvironment{class}{\noindent\footnotesize\textbf{Mathematics subject classification 2000:}}{}

\usepackage{color,tu-preprint}

\newcommand*{\Span}{\operatorname{Span}}            

\newcommand*{\dive}{\operatorname{div}}

\newcommand*{\curl}{\operatorname{curl}}

\newcommand*{\Grad}{\operatorname{Grad}}
\newcommand*{\Div}{\operatorname{Div}}

\newcommand*{\grad}{\operatorname{grad}}

\renewcommand*{\i}{\mathrm{i}}

\DeclareMathAccent{\Circ}{\mathalpha}{operators}{"17}
\newcommand{\interior}[1]{\Circ{#1}}
\renewcommand{\Im}{\operatorname{\mathfrak{Im}}}
\renewcommand{\Re}{\operatorname{\mathfrak{Re}}}

\renewcommand{\hat}{\widehat}
\renewcommand{\tilde}{\widetilde}
\renewcommand*{\epsilon}{\varepsilon}
\renewcommand*{\theta}{\vartheta}
\renewcommand*{\nu}{\varrho}

\makeatother

\usepackage{babel}
  \providecommand{\definitionname}{Definition}
  \providecommand{\examplename}{Example}
  \providecommand{\propositionname}{Proposition}
  \providecommand{\remarkname}{Remark}
\providecommand{\theoremname}{Theorem}

\begin{document}
\setcounter{section}{-1}

\date{}

\title{Mother Operators and their Descendants.%
\thanks{This paper is a corrected update of an extended pre-print of a later
peer-reviewed version of a paper by the same title published originally
only by the first author  in J. Math. Anal. Appl., (1), 403, 54-62,
2013; http://dx.doi.org/10.1016/j.jmaa.2013.02.004%
}}

\author{Rainer Picard, Sascha Trostorff \& Marcus Waurick \thanks{Institut für Analysis, Fachrichtung Mathematik, Technische Universität Dresden, Germany, rainer.picard@tu-dresden.de} }
\maketitle
\begin{abstract}
\noindent \textbf{Abstract.} A mechanism deriving new well-posed
evolutionary equations from given ones is inspected. It turns out
that there is one particular spatial operator from which many of the
standard evolutionary problems of mathematical physics can be generated
by this abstract mechanism using suitable projections. The complexity
of the dynamics of the phenomena considered can be described in terms
of suitable material laws. The idea is illustrated with a number of
concrete examples.
\end{abstract}

\begin{class}
47F05 Partial differential operators, 47N20 Applications to differential and integral equations, 35F05 General theory of linear first-order PDE, 35M10 PDE of mixed type, 46N20 Applications to differential and integral equations 
\end{class}\setcounter{page}{1}

\tableofcontents{}

\setcounter{section}{-1}

\section{Introduction}

In \cite{Pi2009-1,PDE_DeGruyter} it has been shown that the standard
(autonomous and linear) initial boundary value problems of mathematical
physics share a simple common form, if considered as first order systems.
Indeed, it is found that they are of the form
\[
\partial_{0}V+AU=F,
\]
where in the usual cases $A$ is skew-selfadjoint, and $U$ and $V$
are linked by a so-called material law
\[
V=\mathcal{M}\left(\partial_{0}^{-1}\right)U,
\]
where $\mathcal{M}$ is a bounded operator-valued function, analytic
in a ball $B_{\mathbb{C}}\left(r,r\right)$ of radius $r\in\mathbb{R}_{>0}$
centered at $r$. $\mathcal{M}\left(\partial_{0}^{-1}\right)$ is
then well-defined in terms of an operator-valued function calculus
associated with $\partial_{0}$ as a normal operator in the weighted
$L^{2}$-type space $H_{\nu,0}\left(\mathbb{R},H\right)$, $\nu>\frac{1}{2r}$,
with inner product
\[
\left(U,V\right)\mapsto\int_{\mathbb{R}}\left\langle U\left(t\right)|V\left(t\right)\right\rangle _{H}\,\exp\left(-2\nu t\right)\: dt.
\]
We will not need to recall the solution theory of such equations,
(which we like to refer to as ``evolutionary'' as the term ``evolution
equations'' appears to be reserved for a rather special case in this
wider class), since the purpose of this paper is not on well-posedness
issues but on a remarkable even more specific structural similarity
between various equations of mathematical physics. We shall indeed
see, that most of the standard initial boundary value problems of
mathematical physics can be derived from a single spatial differential
operator of the form

\[
A=\left(\begin{array}{cc}
0 & -\nabla^{*}\\
\nabla & 0
\end{array}\right)
\]
with a suitable domain to make $A$ skew-selfadjoint in
\[
H\coloneqq\left(\bigoplus_{k\in\mathbb{N}}L_{k}^{2}\left(\Omega\right)\right)\oplus\left(\bigoplus_{k\in\mathbb{N}}L_{k}^{2}\left(\Omega\right)\right).
\]
Here $\Omega$ is a non-empty open subset of a Riemannian $C_{1,1}$-manifold
$M$ with $\nabla=d\otimes$ denoting the co-variant derivative (Riemannian
connection). The spaces
\[
L_{k}^{2}\left(\Omega\right)
\]
are the completion of (real- or) complex-valued Lipschitz continuous
covariant $k$-tensor fields having compact support in $\Omega$ with
considered in the norm $\left|\:\cdot\:\right|_{k,0}$ induced by
the inner product
\[
\left(\varphi,\psi\right)\mapsto\int_{M}\left\langle \varphi|\psi\right\rangle _{k}\: V,
\]
where $V$ denotes the volume element associated with the Riemannian
metric tensor field $g$ given by the Riemannian structure of the
manifold $M$. Here $\left\langle \varphi|\psi\right\rangle _{k}$
abbreviates the function 
\[
p\mapsto\left\langle \overline{\varphi\left(p\right)}|\psi\left(p\right)\right\rangle _{k,\left(TM_{p}\right)^{*}}
\]
where
\[
\left(\Phi,\Psi\right)\mapsto\left\langle \Phi|\Psi\right\rangle _{k,\left(TM_{p}\right)^{*}}
\]
is the (real) inner product of covariant $k$-tensors on the tangent
space $TM_{p}$ at $p\in M$ and $\overline{\cdots}$ denotes complex
conjugation. Covariant $0$-tensors are simply real numbers and so
$0$-tensors fields are real-valued functions on $M$ and so we let
\[
\left(\Phi,\Psi\right)\mapsto\left\langle \Phi|\Psi\right\rangle _{0,\left(TM_{p}\right)^{*}}\;:=\Phi\Psi.
\]
Since for $k\in\mathbb{N}_{>0}$ covariant $k$-tensors on the tangent
space $TM_{p}$ are elements in the (real) tensor product space $\bigotimes_{k}\left(TM_{p}\right)^{*}$
the inner product is induced by
\[
\left\langle \Phi_{0}\otimes\cdots\otimes\Phi_{k-1}|\Psi_{0}\otimes\cdots\otimes\Psi_{k-1}\right\rangle _{k,\left(TM_{p}\right)^{*}}=\left\langle \Phi_{0}|\Psi_{0}\right\rangle _{\left(TM_{p}\right)^{*}}\;\cdots\;\left\langle \Phi_{k-1}|\Psi_{k-1}\right\rangle _{\left(TM_{p}\right)^{*}}.
\]
In the sense of this inner product $\nabla^{*}=-\dive$ is the formal
adjoint of the co-variant derivative $\nabla.$ 

The complexity of the various physical phenomena has no influence
on the choice of $A$ but is reflected in different material laws%
\footnote{This is the opposite point of view to the ``conservation law'' perspective.%
}. The process of extraction of particular operators from the ``mother''
operator $\left(\partial_{0}\mathcal{M}\left(\partial_{0}^{-1}\right)+A\right)$
is surprisingly simple and amounts to projecting this operator (class)
down to smaller subspaces. The transparency and simplicity of this
construction is rather striking and shows the interconnectedness of
various different physical phenomena if inspected from a mathematical
point of view. This connection becomes obscured if a second order
(or higher order) model is chosen as a starting point. It appears
that a largely misguided pre-occupation with the occurrence of the
Laplacian%
\footnote{This is surely fostered by the comforting regularity properties of
elliptic (and parabolic) differential operators. In our perspective
these have their place when qualitative properties are of prominent
interest, but, as it turns out, have limited importance and are often
distractive for fundamental well-posedness issues. %
} in equations of mathematical physics has lead to a dominance of second-order
equations and systems in various models. As it turns out, however,
the first order approach leads to a unified and transparent access
to a large class (if not all) of typical linear model equations. 

In the applications we shall for sake of simplicity focus on the Cartesian
or periodic case ($M=\mathbb{R}^{n-k}\times\mathbb{T}^{k}$, $n=1,2,3$,
$k=0,1,2$,3, $k\leq n$, with $\mathbb{T}$ being the flat torus
obtained from the unit interval $[-1/2,1/2[$ by ``gluing'' the
end points together (implying periodicity boundary conditions on $]-1/2,1/2[$
in the last $k$ components), which is perfectly sufficient to understand
the reduction mechanism and its applicability.

\section{Mother Operator, Relatives and Descendants}

\subsection{A Construction Mechanism}
\begin{defn}
Let $C:D\left(C\right)\subseteq H_{0}\to H_{1}$ linear, closed and
densely defined and $B:H_{0}\to X$ a continuous linear mapping, $X,\, H_{0},\, H_{1}$
Hilbert spaces. We say $B$ is \emph{compatible with }$C$ if\end{defn}
\begin{itemize}
\item $CB^{*}$ is densely defined (in $X$).\end{itemize}
\begin{thm}
\label{thm:adjont_compatible}Let $C:D\left(C\right)\subseteq H_{0}\to H_{1}$
linear, closed and densely defined and $B:H_{0}\to X$ a continuous
linear mapping, $X,\, H_{0},\, H_{1}$ Hilbert spaces. Moreover, let
$B$ be compatible with $C.$ Then
\[
\left(CB^{*}\right)^{*}=\overline{BC^{*}}.
\]
\end{thm}
\begin{proof}
It is 
\[
CB^{*}\subseteq\left(BC^{*}\right)^{*}.
\]
Let now $u\in D\left(\left(BC^{*}\right)^{*}\right)$ then for $v\in D\left(BC^{*}\right)=D\left(C^{*}\right)$:
\begin{align*}
\left\langle v|\left(BC^{*}\right)^{*}u\right\rangle  & =\left\langle BC^{*}v|u\right\rangle \\
 & =\left\langle C^{*}v|B^{*}u\right\rangle .
\end{align*}
We read off that
\[
B^{*}u\in D\left(C^{**}\right)=D\left(C\right)
\]
and 
\[
CB^{*}u=\left(BC^{*}\right)^{*}u.
\]
Consequently, we have 
\[
CB^{*}=\left(BC^{*}\right)^{*}
\]
and so
\[
\left(CB^{*}\right)^{*}=\left(BC^{*}\right)^{**}=\overline{BC^{*}}.\tag*{{\qedhere}}
\]
\end{proof}
\begin{defn}
Let $C:D\left(C\right)\subseteq H_{0}\to H_{1}$ linear, closed and
densely defined, $X,\, H_{0},\, H_{1}$ Hilbert spaces. Moreover,
let $B_{0}:H_{0}\to X$ be compatible with $C$ and $B_{1}:H_{1}\to Y$
be compatible with $C^{*}.$ Then we call $\overline{B_{1}C}B_{0}^{*}$
the \emph{$\left(B_{0},B_{1}\right)$-relative} (or simply a \emph{relative)}
of $C$. If not both of the mappings $B_{0},\: B_{1}$ are bijections,
then we call $\overline{B_{1}C}B_{0}^{*}$ the \emph{$\left(B_{0},B_{1}\right)$-descendant}
(or simply a \emph{descendant)} of $C$ (and $C$ the \emph{mother
}operator of $\overline{B_{1}C}B_{0}^{*}$). 
\end{defn}
Of particular interest are compatible operators resulting from orthogonal
projectors. We introduce the following schemes of notation:
\begin{defn}
Let $V$ be a closed subspace of a Hilbert space $H$. Then we denote
the orthogonal projector onto $V$ by $P_{V}.$
\end{defn}
~
\begin{defn}
Let $X\oplus Y$ a direct sum of Hilbert spaces $X,\: Y.$ Then we
denote the canonical projections
\[
\begin{array}{cccccc}
X\oplus Y & \to X,\quad & X\oplus Y & \to Y\\
x\oplus y & \mapsto x,\quad & x\oplus y & \mapsto y
\end{array}
\]
by $\pi_{X}$ and $\pi_{Y}$, respectively.
\end{defn}
These notations are employed in the following elementary observation,
which we record without giving the elementary proof.
\begin{prop}
Let $V$ be a closed subspace of a Hilbert space $H$. Then 
\begin{align*}
P_{V} & =\pi_{V}^{*}\pi_{V}.
\end{align*}
and 
\[
\left(1-P_{V}\right)=\pi_{V^{\perp}}^{*}\pi_{V^{\perp}}=P_{V^{\perp}}.
\]
Moreover, $\pi_{V}^{*}$ is the canonical isometric embedding of $V$
in $H$ and
\[
\pi_{V}\pi_{V}^{*}\mbox{ and }\pi_{V^{\perp}}\pi_{V^{\perp}}^{*}
\]
are the identities on $V$ and $V^{\perp},$ respectively. 
\end{prop}
The subspace $V\oplus\left\{ 0\right\} $ of $V\oplus V^{\perp}$
is commonly identified with $V$. We shall, however, prefer to distinguish
$P_{V}$ and $\pi_{V}$, since the latter allows a proper formulation
of reducing an operator equation to a subspace by constructing appropriate
descendants.

\subsection{Evolutionary Operators and their Relatives}

The concepts introduced in the previous section extends naturally
to evolutionary problems as described in the introduction.

To be specific we consider a particular class of evolutionary problems
of the form:

\begin{equation}
\overline{\left(\partial_{0}\mathcal{M}\left(\partial_{0}^{-1}\right)+A\right)}U=F\label{eq:evo-1}
\end{equation}
where
\[
A\coloneqq\left(\begin{array}{cc}
0 & -C^{*}\\
C & 0
\end{array}\right)
\]
and $C:D\left(C\right)\subseteq H_{0}\to H_{1}$ is a closed densely
defined linear operator so that $A$ is skew-selfadjoint in the Hilbert
space $H\coloneqq H_{0}\oplus H_{1}$. We assume for the material
law that there is a $c_{0}\in\mathbb{R}_{>0}$ with
\begin{equation}
\Re\left\langle \chi_{_{\mathbb{R}_{<0}}}\left(m_{0}\right)U|\partial_{0}\mathcal{M}\left(\partial_{0}^{-1}\right)U\right\rangle _{\nu,0,0}\geq c_{0}\left\langle \chi_{_{\mathbb{R}_{<0}}}\left(m_{0}\right)U|U\right\rangle _{\nu,0,0}\label{eq:posdef-2}
\end{equation}
for all sufficiently large $\nu\in\mathbb{R}_{>0}$ and all $U\in D\left(\partial_{0}\right)\subseteq H_{\nu,0}\left(\mathbb{R},H\right)$.
This assumption warrants solvability and causality of the solution
operator, for this variant of the solution theory compare \cite{Picard2010,PAMM:PAMM201110333}.

In order to ensure (\ref{eq:posdef-2}) we may and will assume that
\begin{align*}
\mathcal{M}\left(\partial_{0}^{-1}\right) & =\mathcal{M}_{0}+\partial_{0}^{-1}\mathcal{M}^{\left(1\right)}\left(\partial_{0}^{-1}\right)
\end{align*}
with $\mathcal{M}_{0}\in L(H)$ selfadjoint%
\footnote{We denote by $L(H)$ the space continuous linear operators from $H$
into $H$. Moreover, we shall not notationally distinguish between
an operator $M\in L(H)$ and its (canonical) extension to the Hilbert
space of $H$-valued $H_{\rho,0}$-functions.%
} and $\pi_{\mathcal{M}_{0}\left[H\right]}\mathcal{M}_{0}\pi_{\mathcal{M}_{0}\left[H\right]}^{*},$
$\pi_{\left[\left\{ 0\right\} \right]\mathcal{M}_{0}}\Re\left(\mathcal{M}^{\left(1\right)}\left(\partial_{0}^{-1}\right)\right)\pi_{\left[\left\{ 0\right\} \right]\mathcal{M}_{0}}^{*}$
uniformly strictly positive definite for all sufficiently large $\nu\in\mathbb{R}_{>0}.$
If $\mathcal{M}_{1}\in L(H)$ 
\begin{align}
\mathcal{M}^{\left(1\right)}\left(\partial_{0}^{-1}\right) & =\mathcal{M}_{1}+\mathcal{M}^{\left(2\right)}\left(\partial_{0}^{-1}\right),\label{eq:shape}
\end{align}
$\mathcal{N}\coloneqq\pi_{\left[\left\{ 0\right\} \right]\mathcal{M}_{0}}\Re\left(\mathcal{M}_{1}\right)\pi_{\left[\left\{ 0\right\} \right]\mathcal{M}_{0}}^{*}$
is strictly positive definite and the operator norm on the space $H_{\nu,0}\left(\mathbb{R},\left[\left\{ 0\right\} \right]\mathcal{M}_{0}\right)$
satisfies
\begin{equation}
\left\Vert \left(\sqrt{\mathcal{N}}\right)^{-1}\pi_{\left[\left\{ 0\right\} \right]\mathcal{M}_{0}}\Re\mathcal{M}^{\left(2\right)}\left(\partial_{0}^{-1}\right)\pi_{\left[\left\{ 0\right\} \right]\mathcal{M}_{0}}^{*}\left(\sqrt{\mathcal{N}}\right)^{-1}\right\Vert <1\label{eq:small}
\end{equation}
for sufficiently large $\rho\in\mathbb{R}_{>0}$, then a standard
perturbation argument shows that (\ref{eq:posdef-2}) is maintained.
Since solution theory is not the topic of this paper we will not dwell
on these issues in the following. We merely note that the construction
of relatives and descendants maintains this solvability condition.

We note first that as a by-product of the above we have the following.
\begin{prop}
\label{prop:relative}Let $C:D\left(C\right)\subseteq H_{0}\to H_{1}$
linear, closed and densely defined. Let $B_{0}:H_{0}\to X$ be compatible
with $C$ and $B_{1}:H_{1}\to Y$ be compatible with $C^{*}$. Then
$B_{0}\oplus B_{1}$ is compatible with 
\[
A\coloneqq\left(\begin{array}{cc}
0 & -C^{*}\\
C & 0
\end{array}\right).
\]
Moreover, if $B_{0}^{\ast}$ has a bounded left-inverse and the set
$B_{0}^{\ast}[X]\cap D(C)$ is a core for $\overline{B_{1}C}|_{\overline{B_{0}^{\ast}[X]}}$
then the relative 
\begin{equation}
\overline{\left(B_{0}\oplus B_{1}\right)A}\left(B_{0}\oplus B_{1}\right)^{*}=\left(\begin{array}{cc}
0 & -\overline{B_{0}C^{*}}B_{1}^{*}\\
\overline{B_{1}C}B_{0}^{*} & 0
\end{array}\right)\label{eq:descent1}
\end{equation}
of $A$\textup{\emph{ is skew-selfadjoint in $X\oplus Y$.}} \end{prop}
\begin{proof}
The compatibility of $B_{0}\oplus B_{1}$ with $A$ is clear. To show
that the relative
\[
\left(\begin{array}{cc}
0 & -\overline{B_{0}C^{*}}B_{1}^{*}\\
\overline{B_{1}C}B_{0}^{*} & 0
\end{array}\right)
\]
of $A$ is again skew-selfadjoint, we have to verify that 
\[
\left(\overline{B_{0}C^{\ast}}B_{1}^{\ast}\right)^{\ast}=\overline{B_{1}C}B_{0}^{\ast}.
\]
Using Theorem \ref{thm:adjont_compatible} for $B_{0}C^{*}$ and $B_{1}$
and the relation $\left(\overline{B_{0}C^{*}}\right)^{*}=\left(B_{0}C^{*}\right)^{*}=CB_{0}^{*}$,
we obtain that 
\[
\left(\overline{B_{0}C^{\ast}}B_{1}^{\ast}\right)^{\ast}=\overline{B_{1}\left(B_{0}C^{\ast}\right)^{\ast}}=\overline{B_{1}CB_{0}^{\ast}},
\]
and thus, it suffices to show $\overline{B_{1}CB_{0}^{\ast}}=\overline{B_{1}C}B_{0}^{\ast}.$
Obviously, $\overline{B_{1}CB_{0}^{\ast}}\subseteq\overline{B_{1}C}B_{0}^{\ast}.$
To see the missing inclusion, let $u\in D\left(\overline{B_{1}C}B_{0}^{\ast}\right),$
i.e. $B_{0}^{\ast}u\in D\left(\overline{B_{1}C}\right).$ By assumption,
there exists a sequence $(x_{n})_{n\in\mathbb{N}}$ in $X$ such that
$B_{0}^{\ast}x_{n}\in D(C)$ for each $n\in\mathbb{N}$ and 
\begin{align*}
B_{0}^{\ast}x_{n} & \to B_{0}^{\ast}u,\\
B_{1}CB_{0}^{\ast}x_{n} & \to\overline{B_{1}C}B_{0}^{\ast}u,
\end{align*}
as $n\to\infty.$ Using, that $B_{0}^{\ast}$ has a bounded left inverse,
we derive that $x_{n}\to u$ as $n\to\infty$ and hence, $u\in D\left(\overline{B_{1}CB_{0}^{\ast}}\right)$,
showing the missing inclusion.\end{proof}
\begin{rem}
~\end{rem}
\begin{enumerate}
\item If $B_{0}^{\ast}$ is onto and continuously invertible, then the assumptions
of the latter proposition are trivially satisfied.
\item The structure of $A$ as $\left(\begin{array}{cc}
0 & -C^{*}\\
C & 0
\end{array}\right)$ also implies that the operators $\overline{\left(\partial_{0}\mathcal{M}\left(\partial_{0}^{-1}\right)+A\right)}$
and $\overline{\left(\partial_{0}\tilde{\mathcal{M}}\left(\partial_{0}^{-1}\right)-A\right)}$
with 
\[
\left(\begin{array}{cc}
1 & 0\\
0 & -1
\end{array}\right)\mathcal{M}\left(\partial_{0}^{-1}\right)\left(\begin{array}{cc}
1 & 0\\
0 & -1
\end{array}\right)=\tilde{\mathcal{M}}\left(\partial_{0}^{-1}\right)
\]
are relatives via $\left(\begin{array}{cc}
1 & 0\\
0 & -1
\end{array}\right)$ as a unitary mapping. 
\item The construction of descendants may be repeated but the result will
in general depend on the order in which the steps are carried out.
\end{enumerate}
We note that the proof of Proposition \ref{prop:relative} mainly
relies on the fact that $\overline{B_{1}CB_{0}^{\ast}}=\overline{B_{1}C}B_{0}^{\ast},$
which follows by the additional assumptions on $B_{0}^{\ast}$. To
show that in general, such an equality cannot be expected, we present
the following example due to \cite{Jeremias}.
\begin{example}
Let $H_{0},H_{1}$ be two Hilbert spaces, where $H_{1}$ is assumed
to be separable. In $H_{0}$ we choose two closed, densely defined
operators $A_{0},A_{1}$ such that $A_{0}\subsetneq A_{1}$ and $x_{1}\in D(A_{0})^{\bot_{D_{A_{1}}}}$,
where the ortho-complement is taken in $D(A_{1})$ with respect to
the graph inner product of $A_{1}$. Moreover, let $(y_{n})_{n\in\mathbb{N}}$
a linear independent total sequence in $H_{1}$ with $y_{n}\to0.$
We define the following operator on $H_{1}$ as the linear extension
of the mapping:
\begin{eqnarray*}
R:\left\{ y_{n}\,|\, n\in\mathbb{N}\right\} \subseteq H_{1} & \to & \ell_{2}(\mathbb{N})\\
y_{k} & \mapsto & 2^{k}e_{k},
\end{eqnarray*}
where $e_{k}$ denotes the $k$-th unit vector in $\ell_{2}(\mathbb{N})$.
We denote its extension again by $R$. This operator turns out to
be closable. Indeed, let $\left(z_{n}\right)_{n\in\mathbb{N}}$ be
a sequence in $\Span\left\{ y_{n}\,|\, n\in\mathbb{N}\right\} $ with
$z_{n}\to0$ and $Rz_{n}\to z\in\ell_{2}(\mathbb{N})$ as $n\to\infty$.
For each $n\in\mathbb{N}$ there exists a sequence $(\lambda_{k}^{n})_{k\in\mathbb{N}}$
of complex numbers with almost all entries being zero such that 
\[
z_{n}=\sum_{k=0}^{\infty}\lambda_{k}^{n}y_{k}.
\]
Since $z_{n}\to0$ we obtain $\lambda_{k}^{n}\to0$ as $n\to\infty$
for each $k\in\mathbb{N}.$ The latter yields 
\[
\langle Rz_{n}|e_{j}\rangle_{\ell_{2}}=2^{j}\lambda_{j}^{n}\to0
\]
 as $n\to\infty.$ This shows $z=0$ and thus, $R$ is closable. Consider
now the operator $Q$ defined as 
\begin{eqnarray*}
Q:\Span\left\{ (x_{1},y_{n})\,|\, n\in\mathbb{N}\right\} \subseteq\Span\left\{ x_{1}\right\} \oplus H_{1} & \to & \ell_{2}(\mathbb{N})\\
(w,z) & \mapsto & Rz.
\end{eqnarray*}
Again, this operator is closable, which follows from the closability
of $R$. Then $(x_{1},0)\notin D(\overline{Q}).$ Indeed, let $\left((w_{n},z_{n})\right)_{n\in\mathbb{N}}$
be a sequence in $\Span\left\{ (x_{1},y_{n})\,|\, n\in\mathbb{N}\right\} $
with $w_{n}\to x_{1}$ and $z_{n}\to0.$ For each $n\in\mathbb{N}$
there exists a finite sequence $(\lambda_{k}^{n})_{k\in\mathbb{N}}$
of complex scalars such that 
\[
w_{n}=\sum_{k=0}^{\infty}\lambda_{k}^{n}x_{1}\quad\mbox{and}\quad z_{n}=\sum_{k=0}^{\infty}\lambda_{k}^{n}y_{k}.
\]
Since $z_{n}\to0$ we get $\lambda_{k}^{n}\to0$ as $n\to\infty$
and thus, 
\[
\bigwedge_{m\in\mathbb{N}}\,\bigvee_{n_{0}\in\mathbb{N}}\,\bigwedge_{n\in\mathbb{N}_{\geq n_{0}}}\sum_{k=0}^{m}|\lambda_{k}^{n}|\leq\frac{1}{2^{m}}.
\]
Moreover, since $w_{n}\to x_{1}$ we derive $\sum_{k=0}^{\infty}\lambda_{k}^{n}\to1$
as $n\to\infty,$ which yields 
\[
\bigvee_{n_{1}\in\mathbb{N}}\,\bigwedge_{n\in\mathbb{N}_{\geq n_{1}}}\sum_{k=0}^{\infty}|\lambda_{k}^{n}|\geq\frac{1}{2}.
\]
Let $m\in\mathbb{N}$ and $n\geq\max\{n_{0},n_{1}\}$. Then 
\begin{eqnarray*}
|Rz_{n}|_{\ell_{2}}^{2} & = & \sum_{k=0}^{\infty}2^{2k}\left|\lambda_{k}^{n}\right|^{2}\\
 & = & \sum_{k=0}^{m}2^{2k}\left|\lambda_{k}^{n}\right|^{2}+\sum_{k=m+1}^{\infty}2^{2k}|\lambda_{k}^{n}|^{2}\\
 & \geq & 2^{2\left(m+1\right)}\sum_{k=0}^{\infty}2^{2k}|\lambda_{k+m+1}^{n}|^{2}\\
 & \geq & 2^{2\left(m+1\right)}\frac{3}{4}\left(\sum_{k=m+1}^{\infty}|\lambda_{k}^{n}|\right)^{2}\\
 & = & 2^{2\left(m+1\right)}\frac{3}{4}\left(\sum_{k=0}^{\infty}|\lambda_{k}^{n}|-\sum_{k=0}^{m}|\lambda_{k}^{n}|\right)^{2}\\
 & \geq & 2^{2\left(m+1\right)}\frac{3}{4}\left(\frac{1}{2}-\frac{1}{2^{m}}\right)^{2}\to\infty\quad(m\to\infty).
\end{eqnarray*}
This shows that $\left(Rz_{n}\right)_{n\in\mathbb{N}}$ has an unbounded
subsequence and hence, cannot converge. Thus, $(x_{1},0)\notin D(\overline{Q})$,
which in particular implies that $(x,0)\in D(\overline{Q})$ implies
$x=0.$ \\
Now, we define the operator
\begin{eqnarray*}
A:\left[D(A_{0})\times\{0\}\right]+\left[D(\overline{Q})\right]\subseteq H_{0}\oplus H_{1} & \to & H_{0}\oplus\ell_{2}(\mathbb{N})\\
\left(x,0\right)+(y,z) & \mapsto & \left(A_{1}(x+y),\overline{Q}(y,z)\right).
\end{eqnarray*}
Note that $x+y\in D(A_{1})$ since $y\in\Span\{x_{1}\}.$ We show
that $A$ is closed. To this end, let $\left((x_{n}+y_{n},z_{n})\right)_{n\in\mathbb{N}}$
be a $H_{0}\oplus H_{1}$-convergent sequence in $D(A)$ such that
$\left(A_{1}\left(x_{n}+y_{n}\right)\right)_{n\in\mathbb{N}}$ and
$\left(\overline{Q}(y_{n},z_{n})\right)_{n\in\mathbb{N}}$ converges
in $H_{0}$ and $\ell_{2}(\mathbb{N})$, respectively. Then $\left(x_{n}+y_{n}\right)_{n\in\mathbb{N}}$
converges in $D(A_{1})$ with respect to the graph norm of $A_{1}$,
i.e. in $D_{A_{1}}$, to some $c\in D(A_{1})$. From $x_{n}\in D(A_{0})$
and $y_{n}\in D(A_{0})^{\bot_{D_{A_{1}}}}$ we deduce that both $\left(x_{n}\right)_{n\in\mathbb{N}}$
and $\left(y_{n}\right)_{n\in\mathbb{N}}$ converge in $D_{A_{1}}$
to some $x\in D(A_{1})$ and $y\in D(A_{1})$, respectively. Since
$A_{0}\subseteq A_{1}$ is closed, we infer $x\in D(A_{0})$. Now,
from the convergence of $\left(y_{n}\right)_{n}$, $\left((x_{n}+y_{n},z_{n})\right)_{n\in\mathbb{N}}$
and $\left(\overline{Q}(y_{n},z_{n})\right)_{n\in\mathbb{N}}$ together
with the closedness of $\overline{Q}$, we infer that $\left((y_{n},z_{n})\right)_{n\in\mathbb{N}}$
converges in $D(\overline{Q})$ to $(y,z)\in D(\overline{Q})$ with
respect to the graph norm of $\overline{Q}$ for some $z\in\ell_{2}(\mathbb{N})$.
Summarizing, we have $x_{n}+y_{n}\to x+y$ in $D_{A_{1}}$ as $n\to\infty$
with $x\in D(A_{0})$ and $y\in\Span\{x_{1}\}$ as well as $(y_{n},z_{n})\to(y,z)$
in $D_{\overline{Q}}$ as $n\to\infty$. Thus $A$ is closed. \\
We define 
\begin{eqnarray*}
B:H_{0} & \to & H_{0}\oplus H_{1}\\
x & \mapsto & (x,0)
\end{eqnarray*}
and 
\begin{eqnarray*}
C:H_{0}\oplus\ell_{2}(\mathbb{N}) & \to & H_{0}\\
(x,y) & \mapsto & x.
\end{eqnarray*}
Then 
\[
CAB=A_{0}.
\]
Indeed, let $x\in D(CAB).$ The latter yields that $(x,0)\in D(A).$
By the definition of $A$ we have that $(x-x_{0},0)\in D(\overline{Q})$
for some $x_{0}\in D(A_{0}).$ By what we have shown above, this holds
if and only if $x=x_{0}$ and thus $x\in D(A_{0}).$ The other inclusion
holds trivially. \\
We show that $x_{1}\in D(\overline{CA}B).$ The latter is equivalent
to $(x_{1},0)\in D\left(\overline{CA}\right).$ For each $n\in\mathbb{N}$
we have that $(x_{1},y_{n})\in D(Q)\subseteq D(CA).$ Since $(x_{1},y_{n})\to(x_{1},0)$
and $CA(x_{1},y_{n})=A_{1}x_{1}$ for each $n\in\mathbb{N}$ we obtain
$(x_{1},0)\in D(\overline{CA}).$ Thus, 
\[
\overline{CAB}=CAB\subsetneq\overline{CA}B.
\]
Note that with $H_{1}=\ell_{2}(\mathbb{N})$, we even have $B=C^{*}$
and thus $\overline{CAC^{*}}=CAC^{*}\subsetneq\overline{CA}C^{*}.$
\end{example}
Applying earlier observations to evolutionary operators yields the
following result.
\begin{thm}
Let $C:D\left(C\right)\subseteq H_{0}\to H_{1}$ linear, closed and
densely defined, and let $B_{0}:H_{0}\to X$ be compatible with $C$
and $B_{1}:H_{1}\to Y$ be compatible with $C^{*}$. Moreover, let
$B_{0}^{*}$ have a bounded left inverse and let $B_{0}^{\ast}[X]\cap D(C)$
be a core for $\overline{B_{1}C}|_{\overline{B_{0}^{\ast}[X]}}$ 
\begin{align*}
A & \coloneqq\left(\begin{array}{cc}
0 & -C^{*}\\
C & 0
\end{array}\right).
\end{align*}
Then $\overline{\left(B_{0}\oplus B_{1}\right)A}\left(B_{0}\oplus B_{1}\right)^{*}$
is skew-selfadjoint and 
\[
\overline{\partial_{0}\left(B_{0}\oplus B_{1}\right)\mathcal{M}\left(\partial_{0}^{-1}\right)\left(B_{0}\oplus B_{1}\right)^{*}+\overline{B_{0}\oplus B_{1}A}\left(B_{0}\oplus B_{1}\right)^{*}}
\]
is an \emph{(evolutionary)} relative of $\overline{\left(\partial_{0}\mathcal{M}\left(\partial_{0}^{-1}\right)+A\right)}$. 
\end{thm}
Unitary equivalence of $A$ would be a typical example illustrating
the previous theorem. More interesting, however, are descendants being
produced by projections onto proper subspaces. 
\begin{rem}
In contrast to the last result general relatives of $\overline{\left(\partial_{0}\mathcal{M}\left(\partial_{0}^{-1}\right)+A\right)}$
need not maintain its clear formal structure. Indeed, a few elementary
row and and column operations can produce almost impenetrably confusing
``model equations''. In many instances it turns out to be the main
task to reconstruct $\left(\partial_{0}\mathcal{M}\left(\partial_{0}^{-1}\right)+A\right)$
from a quite different looking (often only a formal) relative. 
\end{rem}

\subsection{Reducing Equations to Subspaces}

\subsubsection{The General Construction}

Let now $V$ be a closed subspace such that 
\[
P_{V}\mathcal{M}\left(\partial_{0}^{-1}\right)=\mathcal{M}\left(\partial_{0}^{-1}\right)P_{V}.
\]
 Then, applying $\pi_{V}$ to equation (\ref{eq:evo-1}) we obtain
similarly as before 
\begin{equation}
\partial_{0}\left(\pi_{V}\mathcal{M}\left(\partial_{0}^{-1}\right)\pi_{V}^{*}\right)\pi_{V}u+\pi_{V}A\left(\pi_{V}^{*}\pi_{V}u+\pi_{V^{\perp}}^{*}\pi_{V^{\perp}}u\right)=\pi_{V}F.\label{eq:evo:cut}
\end{equation}
This is now an equation for $\pi_{V}u$ in $V$ with little chance
of being well-posed (due to the free floating part $\pi_{V^{\perp}}^{*}\pi_{V^{\perp}}u$).
To enforce well-posedness, we could assume that $\pi_{V}$ and $\pi_{V}^{*}$
are compatible with $A$ and instead (assuming $\pi_{V^{\perp}}u=0$)
consider
\[
\left(\overline{\partial_{0}\left(\pi_{V}\mathcal{M}\left(\partial_{0}^{-1}\right)\pi_{V}^{*}\right)+\left(\overline{\pi_{V}A}\pi_{V}^{*}\right)}\right)\pi_{V}u=\pi_{V}F.
\]

In general, however, $\overline{\pi_{V}A}\pi_{V}^{*}$ will fail to
be skew-selfadjoint although 
\[
\pi_{V}\mathcal{M}\left(\partial_{0}^{-1}\right)\pi_{V}^{*}
\]
inherits its positive definiteness property from $\mathcal{M}\left(\partial_{0}^{-1}\right)$:
\\
There is a $c_{0}\in\mathbb{R}_{>0}$ with
\begin{equation}
\begin{array}{rl}
\Re\left\langle \chi_{_{\mathbb{R}_{<0}}}\left(m_{0}\right)U|\partial_{0}\pi_{V}\mathcal{M}\left(\partial_{0}^{-1}\right)\pi_{V}^{*}u\right\rangle _{\nu,0,0} & \geq c_{0}\left\langle \chi_{_{\mathbb{R}_{<0}}}\left(m_{0}\right)\pi_{V}^{*}u|\pi_{V}^{*}u\right\rangle _{\nu,0,0}\\
 & \qquad\qquad=c_{0}\left\langle \chi_{_{\mathbb{R}_{<0}}}\left(m_{0}\right)u|u\right\rangle _{\nu,0,0}
\end{array}\label{eq:posdef-1}
\end{equation}
for all sufficiently large $\nu\in\mathbb{R}_{>0}$ and all $u\in D\left(\partial_{0}\right)\subseteq H_{\nu}\left(\mathbb{R},V\right)$. 

But if $V=V_{0}\oplus V_{1}$ with $V_{0}=H_{0}$ or $V_{1}=H_{1}$,
where $\pi_{V_{0}}$ is compatible with $C$ and $\pi_{V_{1}}$ is
compatible with $C^{*},$ we see from our earlier considerations that
we have an evolutionary descendant with $\overline{\pi_{V}A}\pi_{V}^{*}$
skew-selfadjoint. We summarize this observation in our next theorem.
\begin{thm}
Let $V_{0}\subseteq H_{0}$, $V_{1}\subseteq H_{1}$ be closed subspaces
such that $\pi_{V_{0}}$ is compatible with $C$ and $\pi_{V_{1}}$
is compatible with $C^{*}.$ Then with $V\coloneqq V_{0}\oplus H_{1}$
or $V\coloneqq H_{0}\oplus V_{1}$ we have that the operator $\left(\overline{\partial_{0}\left(\pi_{V}\mathcal{M}\left(\partial_{0}^{-1}\right)\pi_{V}^{*}\right)+\left(\overline{\pi_{V}A}\pi_{V}^{*}\right)}\right)$
is the evolutionary $\pi_{V}$-descendant of $\overline{\left(\partial_{0}\mathcal{M}\left(\partial_{0}^{-1}\right)+A\right)}$.
\end{thm}

\subsubsection{A Particular Case: Removing Null Spaces}

To see the above construction at work let us consider the standard
issue of reducing $A$ to the ortho-complement of its kernel for the
operator 
\[
\left(\partial_{0}\mathcal{M}\left(\partial_{0}^{-1}\right)+A\right).
\]
We assume for sake of definiteness that 
\[
\mathcal{M}\left(\partial_{0}^{-1}\right)=\mathcal{M}_{0}+\partial_{0}^{-1}\mathcal{M}_{1}\left(\partial_{0}^{-1}\right)
\]
for a selfadjoint $\mathcal{M}_{0}\in L(H)$ and a $L(H)$-valued
analytic function $\mathcal{M}_{1}$ for $H$ being the underlying
(spatial) Hilbert space. We have that 
\[
\left(\partial_{0}\left(\pi_{\overline{A\left[H\right]}}\mathcal{M}\left(\partial_{0}^{-1}\right)\pi_{\overline{A\left[H\right]}}^{*}\right)+\pi_{\overline{A\left[H\right]}}A\pi_{\overline{A\left[H\right]}}^{*}\right)
\]
and
\[
\left(\partial_{0}\left(\pi_{\left[\left\{ 0\right\} \right]A}\mathcal{M}\left(\partial_{0}^{-1}\right)\pi_{\left[\left\{ 0\right\} \right]A}^{*}\right)\right)
\]
are two corresponding relatives. It is $\pi_{\overline{A\left[H\right]}}A\pi_{\overline{A\left[H\right]}}^{*}$
skew-selfadjoint, since $\overline{A\left[H\right]}$ is a reducing
subspace of $A$. Note that here $\pi_{\overline{A\left[H\right]}}A$
is already closed and $\overline{\pi_{\left[\left\{ 0\right\} \right]A}A}=0$.
Moreover, 
\[
\pi_{\overline{A\left[H\right]}}=\pi_{\overline{C^{*}\left[H_{1}\right]}}\oplus\pi_{\overline{C\left[H_{0}\right]}}
\]
and 
\[
\pi_{\left[\left\{ 0\right\} \right]A}=\pi_{\left[\left\{ 0\right\} \right]C}\oplus\pi_{\left[\left\{ 0\right\} \right]C^{*}.}
\]
Clearly, $\pi_{\left[\left\{ 0\right\} \right]C}$ and $\pi_{\overline{C^{*}\left[H_{1}\right]}}$
are compatible with $C$ and correspondingly $\pi_{\left[\left\{ 0\right\} \right]C^{*}}$
and $\pi_{\overline{C\left[H_{0}\right]}}$ are compatible with $C^{*}.$
So, if at least one of the null spaces $\left[\left\{ 0\right\} \right]C$
or $\left[\left\{ 0\right\} \right]C^{*}$ is non-trivial, which is
the only interesting case, we have that $\left(\partial_{0}\left(\pi_{\overline{A\left[H\right]}}\mathcal{M}\left(\partial_{0}^{-1}\right)\pi_{\overline{A\left[H\right]}}^{*}\right)+\pi_{\overline{A\left[H\right]}}A\pi_{\overline{A\left[H\right]}}^{*}\right)$
and $\left(\partial_{0}\left(\pi_{\left[\left\{ 0\right\} \right]A}\mathcal{M}\left(\partial_{0}^{-1}\right)\pi_{\left[\left\{ 0\right\} \right]A}^{*}\right)\right)$
are indeed descendants of $\left(\partial_{0}\mathcal{M}\left(\partial_{0}^{-1}\right)+A\right).$
How can these descendants help in solving a problem for the mother
operator?

To simplify calculations we first confirm that we may assume that
$\mathcal{M}_{0}$ can be replaced by $\pi_{\mathcal{M}_{0}\left[H\right]}^{*}\pi_{\mathcal{M}_{0}\left[H\right]}=P_{\mathcal{M}_{0}\left[H\right]}$.
Indeed,
\[
H=\mathcal{M}_{0}\left[H\right]\oplus\left[\left\{ 0\right\} \right]\mathcal{M}_{0}
\]
and 
\[
\mathcal{M}_{0}=\pi_{\mathcal{M}_{0}\left[H\right]}\mathcal{M}_{0}\pi_{\mathcal{M}_{0}\left[H\right]}^{*}\oplus0_{\left[\left\{ 0\right\} \right]\mathcal{M}_{0}}
\]
With
\begin{align*}
\tilde{\mathcal{M}}_{0} & =\pi_{\mathcal{M}_{0}\left[H\right]}\mathcal{M}_{0}\pi_{\mathcal{M}_{0}\left[H\right]}^{*}\oplus\pi_{\left[\left\{ 0\right\} \right]\mathcal{M}_{0}}\pi_{\left[\left\{ 0\right\} \right]\mathcal{M}_{0}}^{*}
\end{align*}
we obtain
\begin{align*}
\sqrt{\tilde{\mathcal{M}}_{0}^{-1}}\mathcal{M}_{0}\sqrt{\tilde{\mathcal{M}}_{0}^{-1}} & =\sqrt{\tilde{\mathcal{M}}_{0}^{-1}}\sqrt{\tilde{\mathcal{M}}_{0}}\left(1_{\mathcal{M}_{0}\left[H\right]}\oplus0_{\left[\left\{ 0\right\} \right]\mathcal{M}_{0}}\right)\left(1_{\mathcal{M}_{0}\left[H\right]}\oplus0_{\left[\left\{ 0\right\} \right]\mathcal{M}_{0}}\right)\sqrt{\tilde{\mathcal{M}}_{0}}\sqrt{\tilde{\mathcal{M}}_{0}^{-1}}\\
 & =\left(1_{\mathcal{M}_{0}\left[H\right]}\oplus0_{\left[\left\{ 0\right\} \right]\mathcal{M}_{0}}\right)=\pi_{\mathcal{M}_{0}\left[H\right]}^{*}\pi_{\mathcal{M}_{0}\left[H\right]}.
\end{align*}
By writing again $A$ for $\sqrt{\tilde{\mathcal{M}}_{0}^{-1}}A\sqrt{\tilde{\mathcal{M}}_{0}^{-1}}$
and $\mathcal{M}^{(1)}\left(\partial_{0}^{-1}\right)$ for $\sqrt{\tilde{\mathcal{M}}_{0}^{-1}}\mathcal{M}^{(1)}\left(\partial_{0}^{-1}\right)\sqrt{\tilde{\mathcal{M}}_{0}^{-1}}$
, (\ref{eq:evo-1}) is recovered but now with a material law operator
of the form 
\[
\mathcal{M}\left(\partial_{0}^{-1}\right)=P_{\mathcal{M}_{0}\left[H\right]}+\partial_{0}^{-1}\mathcal{M}^{\left(1\right)}\left(\partial_{0}^{-1}\right).
\]
Since we have
\[
H=\overline{A\left[H\right]}\oplus\left[\left\{ 0\right\} \right]A
\]
we obtain the decomposition
\begin{align*}
\left(\partial_{0}\left(\pi_{\overline{A\left[H\right]}}\mathcal{M}\left(\partial_{0}^{-1}\right)\pi_{\overline{A\left[H\right]}}^{*}\right)+\pi_{\overline{A\left[H\right]}}A\pi_{\overline{A\left[H\right]}}^{*}\right)\pi_{\overline{A\left[H\right]}}U+\\
+\partial_{0}\left(\pi_{\overline{A\left[H\right]}}\mathcal{M}\left(\partial_{0}^{-1}\right)\pi_{\left[\left\{ 0\right\} \right]A}^{*}\right)\pi_{\left[\left\{ 0\right\} \right]A}U=\pi_{\overline{A\left[H\right]}}F
\end{align*}
and
\[
\left(\partial_{0}\left(\pi_{\left[\left\{ 0\right\} \right]A}\mathcal{M}\left(\partial_{0}^{-1}\right)\pi_{\left[\left\{ 0\right\} \right]A}^{*}\right)\right)\pi_{\left[\left\{ 0\right\} \right]A}U+\left(\partial_{0}\left(\pi_{\left[\left\{ 0\right\} \right]A}\mathcal{M}\left(\partial_{0}^{-1}\right)\pi_{\overline{A\left[H\right]}}^{*}\right)\right)\pi_{\overline{A\left[H\right]}}U=\pi_{\left[\left\{ 0\right\} \right]A}F.
\]
Note that 
\[
\pi_{\overline{A\left[H\right]}}A\pi_{\overline{A\left[H\right]}}^{*}
\]
is now skew-selfadjoint in $\overline{A\left[H\right]}$ and strict
positive definiteness of the real parts of the operators $\partial_{0}\left(\pi_{\overline{A\left[H\right]}}\mathcal{M}\left(\partial_{0}^{-1}\right)\pi_{\overline{A\left[H\right]}}^{*}\right)$
and $\partial_{0}\left(\pi_{\left[\left\{ 0\right\} \right]A}\mathcal{M}\left(\partial_{0}^{-1}\right)\pi_{\left[\left\{ 0\right\} \right]A}^{*}\right)$
is preserved. It should also be clear that here it makes no sense
to assume $\pi_{\left[\left\{ 0\right\} \right]A}u=0$ since the system
would be over-determined since in general we cannot assume that $P_{\left[\left\{ 0\right\} \right]A}$
commutes with $\mathcal{M}\left(\partial_{0}^{-1}\right)$. Instead,
solving the latter equation for $\pi_{\left[\left\{ 0\right\} \right]A}U$
yields
\begin{align}
\pi_{\left[\left\{ 0\right\} \right]A}U & =\left(\partial_{0}\pi_{\left[\left\{ 0\right\} \right]A}\mathcal{M}\left(\partial_{0}^{-1}\right)\pi_{\left[\left\{ 0\right\} \right]A}^{*}\right)^{-1}\pi_{\left[\left\{ 0\right\} \right]A}F+\label{eq:0A}\\
 & -\left(\partial_{0}\pi_{\left[\left\{ 0\right\} \right]A}\mathcal{M}\left(\partial_{0}^{-1}\right)\pi_{\left[\left\{ 0\right\} \right]A}^{*}\right)^{-1}\left(\partial_{0}\pi_{\left[\left\{ 0\right\} \right]A}\mathcal{M}\left(\partial_{0}^{-1}\right)\pi_{\overline{A\left[H\right]}}^{*}\right)\pi_{\overline{A\left[H\right]}}U.\nonumber 
\end{align}
Inserting this into the first equation yields
\begin{align}
 & \left(\partial_{0}\left(\pi_{\overline{A\left[H\right]}}\mathcal{M}\left(\partial_{0}^{-1}\right)\pi_{\overline{A\left[H\right]}}^{*}\right)+\pi_{\overline{A\left[H\right]}}A\pi_{\overline{A\left[H\right]}}^{*}\right)\pi_{\overline{A\left[H\right]}}U+\nonumber \\
 & -\partial_{0}\left(\pi_{\overline{A\left[H\right]}}\mathcal{M}\left(\partial_{0}^{-1}\right)\pi_{\left[\left\{ 0\right\} \right]A}^{*}\right)\left(\partial_{0}\pi_{\left[\left\{ 0\right\} \right]A}\mathcal{M}\left(\partial_{0}^{-1}\right)\pi_{\left[\left\{ 0\right\} \right]A}^{*}\right)^{-1}\left(\partial_{0}\pi_{\left[\left\{ 0\right\} \right]A}\mathcal{M}\left(\partial_{0}^{-1}\right)\pi_{\overline{A\left[H\right]}}^{*}\right)\pi_{\overline{A\left[H\right]}}U\nonumber \\
 & =\pi_{\overline{A\left[H\right]}}F-\partial_{0}\left(\pi_{\overline{A\left[H\right]}}\mathcal{M}\left(\partial_{0}^{-1}\right)\pi_{\left[\left\{ 0\right\} \right]A}^{*}\right)\left(\left(\partial_{0}\pi_{\left[\left\{ 0\right\} \right]A}\mathcal{M}\left(\partial_{0}^{-1}\right)\pi_{\left[\left\{ 0\right\} \right]A}^{*}\right)^{-1}\pi_{\left[\left\{ 0\right\} \right]A}F\right)\nonumber \\
 & =\left(\partial_{0}\tilde{\mathcal{M}}\left(\partial_{0}^{-1}\right)+\pi_{\overline{A\left[H\right]}}A\pi_{\overline{A\left[H\right]}}^{*}\right)\,\pi_{\overline{A\left[H\right]}}U\label{eq:AH}
\end{align}
Now, if {\footnotesize{
\begin{equation}
\begin{array}{rl}
\tilde{\mathcal{M}}\left(\partial_{0}^{-1}\right) & \coloneqq\pi_{\overline{A\left[H\right]}}\mathcal{M}\left(\partial_{0}^{-1}\right)\pi_{\overline{A\left[H\right]}}^{*}+\\
 & -\left(\pi_{\overline{A\left[H\right]}}\mathcal{M}\left(\partial_{0}^{-1}\right)\pi_{\left[\left\{ 0\right\} \right]A}^{*}\right)\left(\partial_{0}\pi_{\left[\left\{ 0\right\} \right]A}\mathcal{M}\left(\partial_{0}^{-1}\right)\pi_{\left[\left\{ 0\right\} \right]A}^{*}\right)^{-1}\left(\partial_{0}\pi_{\left[\left\{ 0\right\} \right]A}\mathcal{M}\left(\partial_{0}^{-1}\right)\pi_{\overline{A\left[H\right]}}^{*}\right)
\end{array}\label{eq:pos-def-tilde-1}
\end{equation}
}}satisfies the required strictly positive definiteness for material
laws in the usual sense in the subspace $\overline{A\left[H\right]}$,
we have solvability in $\overline{A\left[H\right]}$. This observation
for general $A$ may be used to restrict the problem to $\overline{A\left[H\right]}$,
which may allow for example to utilize the compactness of the restricted
resolvent of $A$ in $\overline{A\left[H\right]}$ (if this indeed
holds) whereas the original $A$ may have too large a kernel to have
a compact resolvent, see e.g. \cite{Waurick2012} for an application. 

Assuming (\ref{eq:shape}), (\ref{eq:small}), it suffices to inspect
(\ref{eq:pos-def-tilde-1}) for $\mathcal{M}\left(\partial_{0}^{-1}\right)$
replaced by $\mathcal{M}_{0}+\partial_{0}^{-1}\mathcal{M}_{1}$ due
to the smallness assumption (\ref{eq:small}). Indeed, for the regular
case where $\mathcal{M}_{0}$ is strictly positive, according to the
above, it suffices to consider $\mathcal{M}_{0}=1$ for which 
\begin{align*}
 & \Re\left\langle U|\partial_{0}\left(\pi_{\overline{A\left[H\right]}}\pi_{\overline{A\left[H\right]}}^{*}-\pi_{\overline{A\left[H\right]}}\pi_{\left[\left\{ 0\right\} \right]A}^{*}\left(\pi_{\left[\left\{ 0\right\} \right]A}\pi_{\left[\left\{ 0\right\} \right]A}^{*}\right)^{-1}\pi_{\left[\left\{ 0\right\} \right]A}\pi_{\overline{A\left[H\right]}}^{*}\right)U\right\rangle _{\nu,0,0}=\\
 & =\Re\left\langle U|\partial_{0}\pi_{\overline{A\left[H\right]}}\pi_{\overline{A\left[H\right]}}^{*}U\right\rangle _{\nu,0,0}\\
 & =\nu\Re\left\langle \pi_{\overline{A\left[H\right]}}^{*}U|\pi_{\overline{A\left[H\right]}}^{*}U\right\rangle _{\nu,0,0}=\nu\Re\left\langle U|U\right\rangle _{\nu,0,0}
\end{align*}
for $U\in H_{\nu,1}\left(\mathbb{R},\overline{A\left[H\right]}\right)$,
see also \cite[Theorem 6.11]{Waurick2012b} for the case of $\mathcal{M}_{0}$
having non-trivial nullspace.

Substituting the solution of this standard evolutionary problem into
(\ref{eq:0A}) we obtain the other part of $U$.

\section{Some Applications}

\subsection{A Particular Mother Operator}

Now we consider specifically
\begin{equation}
\left(\partial_{0}\mathcal{M}\left(\partial_{0}^{-1}\right)+A\right)U=F\label{eq:mother-evo}
\end{equation}
where
\begin{equation}
A=\left(\begin{array}{cc}
0 & -\nabla^{*}\\
\nabla & 0
\end{array}\right)\label{eq:mother}
\end{equation}
with a suitable domain making $A$ skew-selfadjoint in the Hilbert
space%
\footnote{We have chosen here to add up tensor spaces of all ranks, although
in applications only $k=0,1,2,3$ appear to be relevant. Restricting
to the ``physically relevant'' subspace 
\[
V=L_{0}^{2}\left(\Omega\right)\oplus L_{1}^{2}\left(\Omega\right)\oplus L_{2}^{2}\left(\Omega\right)\oplus L_{3}^{2}\left(\Omega\right)
\]
may therefore be considered as a first application of the mechanism
to generate descendants of the operator in (\ref{eq:mother-evo},\ref{eq:mother}).%
} 
\[
H=\left(\bigoplus_{k\in\mathbb{N}}L_{k}^{2}\left(\Omega\right)\right)\oplus\left(\bigoplus_{k\in\mathbb{N}}L_{k}^{2}\left(\Omega\right)\right).
\]
Here $\nabla$ and $\nabla^{*}$ are formal adjoints on the linear
subspace 
\[
\left(\bigoplus_{k\in\mathbb{N}}\interior{C}_{1,k}\left(\Omega\right)\right)\oplus\left(\bigoplus_{k\in\mathbb{N}}\interior{C}_{1,k}\left(\Omega\right)\right)
\]
of $H$, where $\interior{C}_{1,k}\left(\Omega\right)$ denotes the
space of $C_{1}$-smooth co-variant tensor fields of rank $k$ with
compact support on a Riemannian $C_{1,1}$-manifold $M$ with metric
tensor $g$. The differential operator $\nabla$ is the so-called
co-variant derivative and its skew-adjoint $-\nabla^{*}$ is frequently
introduced as the tensorial divergence%
\footnote{Correspondingly, we could use $\grad$ as a notation for the covariant
derivative $\nabla$ to give the original evolutionary equation the
suggestive look of the acoustic system (see below).%
} $\dive$. For sake of definiteness we shall only consider the choice%
\footnote{This choice may be referred to as the Dirichlet boundary condition
case. %
}
\[
A\coloneqq\left(\begin{array}{cc}
0 & -\left(\interior{\nabla}\right)^{*}\\
\interior{\nabla} & 0
\end{array}\right)
\]

where $\interior{\nabla}$ denotes the closure of $\nabla$ applied
to elements of $\bigoplus_{k\in\mathbb{N}}\interior{C}_{1,k}\left(\Omega\right)$
as an operator in $\bigoplus_{k\in\mathbb{N}}L_{k}^{2}\left(\Omega\right).$
For the material law we impose the usual constraint (\ref{eq:posdef-2}). 

It may be surprising that the majority of initial boundary value problems
from classical mathematical physics can be produced precisely from
(\ref{eq:mother-evo},\ref{eq:mother}) by choosing suitable projections
for constructing descendants. This is the main application of the
above considerations. In order to make matters more easily digestible
we constrain the illustration of our observations to the simple flat
case, i.e. $M$ is $\mathbb{R}^{\ensuremath{n-k}}\times\mathbb{T}^{k}$,
$n=1,2,3,$ $k=0,1,2$,3, $k\leq n$, where $\mathbb{T}$ is the flat
Torus. 

It will turn out that the physical interpretation has remarkably little
relevance for our structural observation. In fact, problems can be
very different in physical interpretation, sharing the same formal
structure makes the solution theory coincide.

\subsection{Isolated Physical Phenomena}

\subsubsection{\label{sub:Acoustic-Equation--Heat}Acoustic Equation, Heat Conduction
and the Relativistic Schr\"odinger Equations}

If we choose
\[
V=L_{0}^{2}\left(\Omega\right)\oplus L_{1}^{2}\left(\Omega\right),\:\Omega\subseteq\mathbb{R}^{3},
\]
for our construction of descendants via projectors then we obtain
the classical system governing acoustic waves or, merely depending
on the choice of material law, the heat equation. In the Cartesian
case we have by identifying $0-$tensors with functions and $1-$tensors
with vector fields the classical first order system
\begin{equation}
\left(\partial_{0}\mathcal{M}\left(\partial_{0}^{-1}\right)+\left(\begin{array}{cc}
0 & \dive\\
\interior{\grad} & 0
\end{array}\right)\right)\left(\begin{array}{c}
p\\
v
\end{array}\right)=\left(\begin{array}{c}
f\\
g
\end{array}\right)\label{eq:acoustics-heat}
\end{equation}
with, for example,
\[
\mathcal{M}\left(\partial_{0}^{-1}\right)=\left(\begin{array}{cc}
\rho & 0\\
0 & \kappa
\end{array}\right)+\partial_{0}^{-1}\left(\begin{array}{cc}
0 & 0\\
0 & \sigma
\end{array}\right)
\]

as a simple material law operator. Note that $\overline{P_{L_{0}^{2}\left(\Omega\right)}\left(\interior{\nabla}\right)^{*}}\pi_{L_{1}^{2}\left(\Omega\right)}^{*}=P_{L_{0}^{2}\left(\Omega\right)}\left(\interior{\nabla}\right)^{*}\pi_{L_{1}^{2}\left(\Omega\right)}^{*}=\left(\interior{\nabla}\right)^{*}\pi_{L_{1}^{2}\left(\Omega\right)}^{*}$and
$\overline{P_{L_{0}^{1}\left(\Omega\right)}\interior{\nabla}}\pi_{L_{0}^{2}\left(\Omega\right)}^{*}=P_{L_{0}^{1}\left(\Omega\right)}\interior{\nabla}\pi_{L_{0}^{2}\left(\Omega\right)}^{*}=\interior{\nabla}\pi_{L_{0}^{2}\left(\Omega\right)}^{*}$
are already closed and $\pi_{L_{0}^{2}\left(\Omega\right)}^{*}$,
$\pi_{L_{1}^{2}\left(\Omega\right)}^{*}$ are isometric embeddings
so that
\begin{align*}
\left(\begin{array}{cc}
0 & \dive\\
\interior{\grad} & 0
\end{array}\right)\coloneqq\pi_{V}A\pi_{V}^{*} & =\left(\begin{array}{cc}
0 & -\pi_{L_{0}^{2}\left(\Omega\right)}\left(\interior{\nabla}\right)^{*}\pi_{L_{1}^{2}\left(\Omega\right)}^{*}\\
\pi_{L_{1}^{2}\left(\Omega\right)}\interior{\nabla}\pi_{L_{0}^{2}\left(\Omega\right)}^{*} & 0
\end{array}\right)\\
 & =\left(\begin{array}{cc}
0 & -\overline{\pi_{L_{0}^{2}\left(\Omega\right)}\left(\interior{\nabla}\right)^{*}}\pi_{L_{1}^{2}\left(\Omega\right)}^{*}\\
\overline{\pi_{L_{1}^{2}\left(\Omega\right)}\interior{\nabla}\pi_{L_{0}^{2}\left(\Omega\right)}^{*}} & 0
\end{array}\right)=\overline{\pi_{V}A}\pi_{V}^{*}
\end{align*}
is skew-selfadjoint. If $\rho,\kappa$ are strictly positive definite,
continuous, selfadjoint operators then the material law can be taken
to describe acoustic wave propagation. If $\sigma$ is also strictly
positive definite, continuous and selfadjoint operators, we interpret
$\sigma$ as a damping term. Alternatively this could then be considered
as describing heat propagation with Cattaneo modification. Keeping
all these constraints except for assuming $\kappa=0$, we get the
classical heat propagation. The second row describes in this interpretation
(for $g=0$) the so-called Fourier law of heat conduction. In both
cases, the materials are indeed such that the material law commutes
with complex conjugation. This allows to interpret the equation in
real-valued terms. 

Alternatively, we may view $L_{k}^{2}\left(\Omega\right)=L_{k}^{2}\left(\Omega,\mathbb{C}\right)$,
$k\in\mathbb{N},$ as a Hilbert space over the field $\mathbb{R}$
(by restricting the underlying scalar field). Then we have
\begin{align*}
\mathcal{R}:L_{k}^{2}\left(\Omega,\mathbb{C}\right) & \to L_{k}^{2}\left(\Omega,\mathbb{R}\right)\oplus L_{k}^{2}\left(\Omega,\mathbb{R}\right)\\
u & \mapsto\left(\begin{array}{c}
\Re u\\
\Im u
\end{array}\right)
\end{align*}
as an $\mathbb{R}$-unitary mapping. For example multiplication by
the complex unit is then unitarily equivalent to
\[
\mathcal{R}\i\mathcal{R}^{-1}=\left(\begin{array}{cc}
0 & -1\\
1 & 0
\end{array}\right).
\]

With this observation we get that the Schr\"odinger operator $\partial_{0}+\i\,\Delta_{D}$
assumes the real form 
\begin{equation}
\partial_{0}+\left(\begin{array}{cc}
0 & -\Delta_{D}\\
\Delta_{D} & 0
\end{array}\right),\label{eq:real-schroe}
\end{equation}
where, to be specific about boundary conditions, we have chosen the
Dirichlet-Laplacian $\Delta_{D}.$ Clearly, due to its second order
type this operator is not covered in our approach. There is, however,
a variant known as the relativistic Schr\"odinger operator in which
$-\Delta_{D}$ is simply replaced by $\sqrt{-\Delta_{D}}=\left|\interior{\grad}\right|,$
which turns out to be essentially unitarily equivalent to the acoustics
problem (\ref{eq:acoustics-heat}) for an even simpler material law.

Indeed, removing the null space of $A=\left(\begin{array}{cc}
0 & \dive\\
\interior{\grad} & 0
\end{array}\right)$ by reducing further to the subspace $L^{2}\left(\Omega\right)\oplus\overline{\interior{\grad}\left[L^{2}\left(\Omega\right)\right]},$
which is the range of $A,$ we obtain another descendant of (\ref{eq:mother-evo},\ref{eq:mother}),
which is indeed a relative of the relativistic Schr\"odinger operator.
According to the polar decomposition theorem there is a unitary mapping
$U$ such that
\[
\interior{\grad}=U\left|\interior{\grad}\right|
\]
 and 
\[
-\dive=\left|\interior{\grad}\right|U^{*}.
\]
Consequently,
\begin{align*}
 & \left(\partial_{0}\left(\left(\begin{array}{cc}
1 & 0\\
0 & U
\end{array}\right)\mathcal{M}\left(\partial_{0}^{-1}\right)\left(\begin{array}{cc}
1 & 0\\
0 & U^{*}
\end{array}\right)\right)+\left(\begin{array}{cc}
0 & \dive\\
\interior{\grad} & 0
\end{array}\right)\right)=\\
 & =\left(\begin{array}{cc}
1 & 0\\
0 & U
\end{array}\right)\left(\partial_{0}\mathcal{M}\left(\partial_{0}^{-1}\right)+\left(\begin{array}{cc}
0 & -\left|\interior{\grad}\right|\\
\left|\interior{\grad}\right| & 0
\end{array}\right)\right)\left(\begin{array}{cc}
1 & 0\\
0 & U^{*}
\end{array}\right).
\end{align*}
Since for the relativistic Schr\"odinger operator $\mathcal{M}\left(\partial_{0}^{-1}\right)=1,$
we get
\[
\left(\partial_{0}+\left(\begin{array}{cc}
0 & \dive\\
\interior{\grad} & 0
\end{array}\right)\right)=\left(\begin{array}{cc}
1 & 0\\
0 & U
\end{array}\right)\left(\partial_{0}+\left(\begin{array}{cc}
0 & -\left|\interior{\grad}\right|\\
\left|\interior{\grad}\right| & 0
\end{array}\right)\right)\left(\begin{array}{cc}
1 & 0\\
0 & U^{*}
\end{array}\right).
\]

There is another rather common version to translate the operator of
the wave equation into a first order in time system, which, however,
is nothing but another relative of the acoustics operator. Utilizing
the naive analogy to the ordinary differential equations case one
translates
\[
\partial_{0}^{2}u-\Delta_{D}u=f
\]
into the system
\begin{align*}
\left(\partial_{0}+\left(\begin{array}{cc}
0 & \epsilon\\
0 & 0
\end{array}\right)+\left(\begin{array}{cc}
0 & \Delta_{D}-\epsilon\\
1 & 0
\end{array}\right)\right)\left(\begin{array}{c}
\partial_{0}u\\
-u
\end{array}\right) & =\left(\partial_{0}+\left(\begin{array}{cc}
0 & \Delta_{D}\\
1 & 0
\end{array}\right)\right)\left(\begin{array}{c}
\partial_{0}u\\
-u
\end{array}\right)\\
 & =\left(\begin{array}{c}
f\\
0
\end{array}\right)
\end{align*}

Here we choose $-\epsilon\in\mathbb{R}_{\leq0}$ in the resolvent
set of $-\Delta_{D}.$ The reasoning goes like this:
\[
\left(\begin{array}{cc}
0 & \Delta_{D}-\epsilon\\
1 & 0
\end{array}\right)
\]
is skew-selfadjoint considered in $L^{2}\left(\Omega\right)\oplus H_{1}\left(\sqrt{-\Delta_{D}+\epsilon}\right)$
where $H_{1}\left(\sqrt{-\Delta_{D}+\epsilon}\right)$ is $D\left(\sqrt{-\Delta_{D}}\right)$
equipped with the inner product
\[
\left(u,v\right)\mapsto\left\langle \sqrt{-\Delta_{D}}u|\sqrt{-\Delta_{D}}v\right\rangle _{0}+\epsilon\left\langle u|v\right\rangle _{0},
\]
i.e. for $\epsilon=1$ the graph inner product of $\sqrt{-\Delta_{D}}$.
We use that 
\[
\sqrt{-\Delta_{D}+\epsilon}:H_{1}\left(\sqrt{-\Delta_{D}+\epsilon}\right)\to L^{2}\left(\Omega\right)
\]
is unitary.

Now, we aim to show that the more general system 
\[
\partial_{0}\mathcal{M}\left(\partial_{0}^{-1}\right)+\left(\begin{array}{cc}
0 & \Delta_{D}\\
1 & 0
\end{array}\right)
\]
is indeed a relative to a first-order-in-time-and-space-system. We
note that
\[
\partial_{0}\mathcal{M}\left(\partial_{0}^{-1}\right)+\left(\begin{array}{cc}
0 & \Delta_{D}\\
1 & 0
\end{array}\right)=\partial_{0}\left(\mathcal{M}\left(\partial_{0}^{-1}\right)+\partial_{0}^{-1}\left(\begin{array}{cc}
0 & \epsilon\\
0 & 0
\end{array}\right)\right)+\left(\begin{array}{cc}
0 & \Delta_{D}-\epsilon\\
1 & 0
\end{array}\right)
\]
and
\begin{align*}
 & \left(\begin{array}{cc}
1 & 0\\
0 & \sqrt{-\Delta_{D}+\epsilon}
\end{array}\right)\left(\begin{array}{cc}
0 & \Delta_{D}-\epsilon\\
1 & 0
\end{array}\right)\left(\begin{array}{cc}
1 & 0\\
0 & \sqrt{-\Delta_{D}+\epsilon}^{-1}
\end{array}\right)=\\
 & =\left(\begin{array}{cc}
0 & -\sqrt{-\Delta_{D}+\epsilon}\\
\sqrt{-\Delta_{D}+\epsilon} & 0
\end{array}\right)\\
 & =\left(\begin{array}{cc}
1 & 0\\
0 & U_{-}
\end{array}\right)\left(\begin{array}{cc}
0 & -\left|\interior{\grad}\right|+\i\sqrt{\epsilon}\\
\left|\interior{\grad}\right|+\i\sqrt{\epsilon} & 0
\end{array}\right)\left(\begin{array}{cc}
1 & 0\\
0 & U_{+}
\end{array}\right),
\end{align*}
where $\sqrt{-\Delta_{D}+\epsilon}=\left|\left|\interior{\grad}\right|\pm\i\sqrt{\epsilon}\right|$
and the polar decomposition for $\left|\interior{\grad}\right|\pm\i\sqrt{\epsilon}$ 

\[
\left|\interior{\grad}\right|\pm\i\sqrt{\epsilon}=U_{\pm}\left|\left|\interior{\grad}\right|\pm\i\sqrt{\epsilon}\right|
\]
implies 
\begin{align*}
U_{\pm} & =\left(\left|\interior{\grad}\right|\pm\i\sqrt{\epsilon}\right)\left|\left|\interior{\grad}\right|\pm\i\sqrt{\epsilon}\right|^{-1}\\
 & =\left|\interior{\grad}\right|\left|\left|\interior{\grad}\right|\pm\i\sqrt{\epsilon}\right|^{-1}\pm\i\sqrt{\epsilon}\left|\left|\interior{\grad}\right|\pm\i\sqrt{\epsilon}\right|^{-1}
\end{align*}
and 
\[
U_{\pm}^{*}=U_{\mp}.
\]

Using the polar decomposition of $\interior{\grad}=U\left|\interior{\grad}\right|$
we get
\begin{align*}
 & \left(\begin{array}{cc}
1 & 0\\
0 & U
\end{array}\right)\left(\begin{array}{cc}
0 & -\left|\interior{\grad}\right|+\i\sqrt{\epsilon}\\
\left|\interior{\grad}\right|+\i\sqrt{\epsilon} & 0
\end{array}\right)\left(\begin{array}{cc}
1 & 0\\
0 & U^{*}
\end{array}\right)=\\
 & =\left(\begin{array}{cc}
0 & \dive+\i\sqrt{\epsilon}U^{*}\\
\interior{\grad}+\i\sqrt{\epsilon}U & 0
\end{array}\right)\\
 & =\left(\begin{array}{cc}
0 & \dive\\
\interior{\grad} & 0
\end{array}\right)+\i\sqrt{\epsilon}\left(\begin{array}{cc}
0 & U^{*}\\
U & 0
\end{array}\right).
\end{align*}
Thus we obtain for the transformed equation a new material law:
\begin{align*}
\tilde{\mathcal{M}}\left(\partial_{0}^{-1}\right) & =\left(\begin{array}{cc}
1 & 0\\
0 & UU_{+}\sqrt{-\Delta_{D}+\epsilon}
\end{array}\right)\mathcal{M}\left(\partial_{0}^{-1}\right)\left(\begin{array}{cc}
1 & 0\\
0 & \sqrt{-\Delta_{D}+\epsilon}^{-1}U_{-}U^{*}
\end{array}\right)+\\
 & +\partial_{0}^{-1}\left(\begin{array}{cc}
0 & \epsilon\sqrt{-\Delta_{D}+\epsilon}^{-1}U_{-}U^{*}+\i\sqrt{\epsilon}U^{*}\\
\i\sqrt{\epsilon}U & 0
\end{array}\right)
\end{align*}

Given the complexity of the arguments needed it is still a surprisingly
common mechanism used (at least in the case $\epsilon=0$ and simple
material laws) to turn partial differential equations of wave equation
type into first-order-in-time systems. None the less, in the above
terminology we encounter here mere relatives of the system of the
acoustic equations, which in turn is a descendant of (\ref{eq:mother-evo},\ref{eq:mother}).

\subsubsection{Elastic Waves}

If we choose
\[
V=L_{1}^{2}\left(\Omega\right)\oplus\mathrm{sym}\left[L_{2}^{2}\left(\Omega\right)\right]
\]
with $\Omega$ open and non-empty in $\mathbb{R}^{3}$ then we obtain
by analogous arguments the classical system governing the propagation
of waves in elastic or, depending on the choice of material law, viscoelastic
waves. Here $\mathrm{sym}$ is the mapping $\mathrm{sym}:L_{2}^{2}\left(\Omega\right)\to L_{2}^{2}\left(\Omega\right)$
induced by the symmetrization operation for co-variant tensors of
rank 2 
\[
T\mapsto\left(\left(x,y\right)\mapsto\frac{1}{2}\left(T\left(x,y\right)+T\left(y,x\right)\right)\right).
\]
Thus, in Cartesian coordinates the elasticity operator is
\[
\partial_{0}\mathcal{M}\left(\partial_{0}^{-1}\right)+\left(\begin{array}{cc}
0 & \Div\\
\interior{\Grad} & 0
\end{array}\right),
\]
where
\[
\interior{\Grad}v=\frac{1}{2}\left(\partial_{i}v_{j}+\partial_{j}v_{i}\right)_{i,j=1,2,3}
\]
and its negative adjoint
\[
\Div T=\dive T=\left(\sum_{j=1}^{3}\partial_{j}T_{ij}\right)_{i=1,2,3}
\]
 for suitable displacement velocities $v=\left(v_{j}\right)_{j=1,2,3}$
and symmetric stress tensors $\left(T_{ij}\right)_{i,j=1,2,3}.$ A
discussion of various possible material laws of interest can be found
in \cite{Pi2009-1}, compare \cite{pre05172699}.

\subsubsection{Electro-Magnetic Waves}

If we choose
\[
V=L_{1}^{2}\left(\Omega\right)\oplus\mathrm{asym}\left[L_{2}^{2}\left(\Omega\right)\right]
\]
with $\Omega$ open and non-empty in $\mathbb{R}^{3}$ then we obtain
by an analogous reasoning the classical system governing the propagation
of electro-magnetic waves. Here $\mathrm{asym}$ is the mapping $\mathrm{asym}:L_{2}^{2}\left(\Omega\right)\to L_{2}^{2}\left(\Omega\right)$
induced by the anti-symmetrization operation for co-variant tensors
of rank 2 
\[
T\mapsto\left(\left(x,y\right)\mapsto\frac{1}{2}\left(T\left(x,y\right)-T\left(y,x\right)\right)\right).
\]
To convince ourselves that this leads to Maxwell's equations we calculate
$\nabla\cdot W$ for anti-symmetric tensors $W\in\interior C_{\infty,2}\left(\Omega\right)$
in Cartesian coordinates, $\Omega$ being a non-empty open subset
of $\mathbb{R}^{3}$. Let%
\footnote{$W=\mathrm{asym}\left(T_{ks}e^{k}\otimes e^{s}\right)=\frac{1}{2}\left(T_{ks}e^{k}\otimes e^{s}-T_{ks}e^{s}\otimes e^{k}\right)=T_{ks}e^{k}\wedge e^{s}=2\,\sum_{k<s\mod3}T_{ks}e^{k}\wedge e^{s}$%
}
\[
W=\omega_{1}e^{2}\otimes e^{3}+\omega_{2}e^{3}\otimes e^{1}+\omega_{3}e^{1}\otimes e^{2}-\omega_{1}e^{3}\otimes e^{2}-\omega_{2}e^{1}\otimes e^{3}-\omega_{3}e^{2}\otimes e^{1}
\]
then
\[
T_{k,k+1}=\omega_{k+2}e^{k}\otimes e^{k+1}\quad(k\in\{1,2,3\}\text{ with addition }\mod3\text{ and }0\equiv3)
\]
\begin{align*}
\nabla\cdot W & =\partial_{k}W_{ks}e^{s}\\
 & =\partial_{1}W_{12}e^{2}+\partial_{1}W_{13}e^{3}+\partial_{2}W_{23}e^{3}+\partial_{2}W_{21}e^{1}+\partial_{3}W_{32}e^{2}+\partial_{3}W_{31}e^{1}\\
 & =\partial_{1}\omega_{3}e^{2}-\partial_{1}\omega_{2}e^{3}+\partial_{2}\omega_{1}e^{3}-\partial_{2}\omega_{3}e^{1}-\partial_{3}\omega_{1}e^{2}+\partial_{3}\omega_{2}e^{1}\\
 & =\left(\partial_{1}\omega_{3}-\partial_{3}\omega_{1}\right)e^{2}+\left(\partial_{2}\omega_{1}-\partial_{1}\omega_{2}\right)e^{3}+\left(\partial_{3}\omega_{2}-\partial_{2}\omega_{3}\right)e^{1}\\
 & =-\curl\left(\begin{array}{c}
\omega_{1}\\
\omega_{2}\\
\omega_{3}
\end{array}\right).
\end{align*}
Correspondingly,
\[
V=\eta_{1}e^{1}+\eta_{2}e^{2}+\eta_{3}e^{3}
\]
and so
\begin{align*}
\nabla V & =\partial_{1}\eta_{1}e^{1}\otimes e^{1}+\partial_{1}\eta_{2}e^{1}\otimes e^{2}+\partial_{1}\eta_{3}e^{1}\otimes e^{3}+\partial_{2}\eta_{1}e^{2}\otimes e^{1}+\\
 & +\partial_{2}\eta_{2}e^{2}\otimes e^{2}+\partial_{2}\eta_{3}e^{2}\otimes e^{3}+\partial_{3}\eta_{1}e^{3}\otimes e^{1}+\partial_{3}\eta_{2}e^{3}\otimes e^{2}+\partial_{3}\eta_{3}e^{3}\otimes e^{3}.
\end{align*}
We see that 
\begin{align*}
 & \mathrm{asym}\left(\nabla V\right)=\\
 & =\partial_{1}\eta_{2}\frac{1}{2}\left(e^{1}\otimes e^{2}-e^{2}\otimes e^{1}\right)+\partial_{1}\eta_{3}\frac{1}{2}\left(e^{1}\otimes e^{3}-e^{3}\otimes e^{1}\right)+\partial_{2}\eta_{1}\frac{1}{2}\left(e^{2}\otimes e^{1}-e^{1}\otimes e^{2}\right)+\\
 & +\partial_{2}\eta_{3}\frac{1}{2}\left(e^{2}\otimes e^{3}-e^{3}\otimes e^{2}\right)+\partial_{3}\eta_{1}\frac{1}{2}\left(e^{3}\otimes e^{1}-e^{1}\otimes e^{3}\right)+\partial_{3}\eta_{2}\frac{1}{2}\left(e^{3}\otimes e^{2}-e^{2}\otimes e^{3}\right)\\
 & =\left(\partial_{1}\eta_{2}-\partial_{2}\eta_{1}\right)\frac{1}{2}\left(e^{1}\otimes e^{2}-e^{2}\otimes e^{1}\right)+\left(\partial_{2}\eta_{3}-\partial_{3}\eta_{2}\right)\frac{1}{2}\left(e^{2}\otimes e^{3}-e^{3}\otimes e^{2}\right)+\\
 & +\left(\partial_{3}\eta_{1}-\partial_{1}\eta_{3}\right)\frac{1}{2}\left(e^{3}\otimes e^{1}-e^{1}\otimes e^{3}\right)\\
 & =\left(\partial_{1}\eta_{2}-\partial_{2}\eta_{1}\right)\: e^{1}\wedge e^{2}+\left(\partial_{2}\eta_{3}-\partial_{3}\eta_{2}\right)\: e^{2}\wedge e^{3}+\left(\partial_{3}\eta_{1}-\partial_{1}\eta_{3}\right)\: e^{3}\wedge e^{1}\\
 & =d\wedge\eta
\end{align*}

and so 
\[
\overline{\mathrm{asym}\interior\nabla}\eqqcolon\interior d\wedge
\]
on differentiable $1$-form fields. In Cartesian coordinates Maxwell's
equations assume the familiar form
\[
\partial_{0}\mathcal{M}\left(\partial_{0}^{-1}\right)+\left(\begin{array}{cc}
0 & \curl\\
\interior{\curl} & 0
\end{array}\right),
\]
where
\[
\interior{\curl}^{*}=\curl
\]
and containment of $E$ in $D\left(\interior{\curl}\right)$ encodes
and generalizes the electric boundary condition, i.e. vanishing of
the tangential components of $E$, for the electric field $E$ to
the arbitrary boundary of $\Omega$.

\subsubsection{\label{sub:Reducing-Dimensions}Reducing Dimensions}

Another instance of the reduction procedure under discussion is the
reduction of the dimension. We consider the simple case $\Omega\coloneqq\Omega_{0}\times\mathbb{T}^{s}\subseteq\mathbb{R}^{n+1}\times\mathbb{T}^{s}\eqqcolon M$
and want to describe the reduction process from $k$-tensor $L_{k}^{2}(\Omega)$
to $k$-tensor $L_{k}^{2}(\Omega_{0}).$ Throughout, we assume that
the Riemannian metric $g$ only depends on $\Omega_{0},$ that is,
we assume that $g_{ij}=0$ for $i\ne j$ and $g_{ii}=1$ for every
$i,j\in\left\{ n+1,\ldots,n+s\right\} .$ Using Cartesian coordinates,
a covariant $k$-tensor $T\in L_{k}^{2}(\Omega)$ can be written as
\begin{align*}
T & =\sum_{\alpha\in\left\{ 0,\ldots,n+s\right\} ^{k}}\omega_{\alpha}dx^{\alpha}
\end{align*}
for suitable functions $\omega_{\alpha}\in L^{2}(\Omega).$ We set
$\Pi_{n}\coloneqq\left\{ \left.\alpha\in\left\{ 0,\ldots,n+s\right\} ^{k}\right|\bigwedge_{i\in\{0,\ldots,k-1\}}\alpha_{i}\in\left\{ 0,\ldots,n\right\} \right\} $
and define 
\[
\pi_{\Omega_{0}}^{k}:L_{k}^{2}(\Omega)\to L_{k}^{2}(\Omega_{0})
\]
 by 
\begin{align*}
\left(\pi_{\Omega_{0}}^{k}T\right)(t_{0},\ldots,t_{n})\coloneqq & \sum_{\alpha\in\Pi_{n}}\intop_{\left[-\frac{1}{2},\frac{1}{2}\right]}\cdots\intop_{\left[-\frac{1}{2},\frac{1}{2}\right]}\omega_{\alpha}(t_{0},\ldots,t_{n},r_{0},\ldots,r_{s-1})\: dr_{s-1}\ldots dr_{0}\; dx^{\alpha}.
\end{align*}
The adjoint $\left(\pi_{\Omega_{0}}^{k}\right)^{\ast}$ is then the
canonical embedding of $L_{k}^{2}(\Omega_{0})$ into $L_{k}^{2}(\Omega)$
given by 
\[
\left(\pi_{\Omega_{0}}^{k}\right)^{\ast}\left(\sum_{\beta\in\{0,\ldots,n\}^{k}}\psi_{\beta}dx^{\beta}\right)=\sum_{\alpha\in\left\{ 0,\ldots,n+s\right\} ^{k}}\tilde{\psi}_{\alpha}dx^{\alpha},
\]
where 
\[
\tilde{\psi}_{\alpha}(t_{0},\ldots,t_{n},r_{0},\ldots,r_{s-1})\coloneqq\begin{cases}
\psi_{\alpha}(t_{0},\ldots,t_{n}) & \mbox{ if }\alpha\in\Pi_{n},\\
0 & \mbox{ otherwise}.
\end{cases}
\]
Indeed, for $T=\sum_{\alpha\in\left\{ 0,\ldots,n+s\right\} ^{k}}\omega_{\alpha}dx^{\alpha}\in L_{k}^{2}(\Omega)$
and $S=\sum_{\beta\in\{0,\ldots,n\}^{k}}\psi_{\beta}dx^{\beta}\in L_{k}^{2}(\Omega_{0})$
we compute 
\begin{align*}
\langle\pi_{\Omega_{0}}^{k}T|S\rangle_{L_{k}^{2}(\Omega_{0})} & =\sum_{\alpha\in\Pi_{n}}\sum_{\beta\in\{0,\ldots,n\}^{k}}\intop_{\Omega_{0}}\intop_{\left[-\frac{1}{2},\frac{1}{2}\right]}\cdots\intop_{\left[-\frac{1}{2},\frac{1}{2}\right]}\omega_{\alpha}(\cdot,r_{0},\ldots,r_{s-1})\: dr_{s-1}\ldots dr_{0}\psi_{\beta}g^{\alpha\beta}\, dV_{\mathbb{R}^{n+1}}\\
 & =\sum_{\alpha\in\Pi_{n}}\sum_{\beta\in\Pi_{n}}\intop_{\Omega}\omega_{\alpha}\tilde{\psi}_{\beta}g^{\alpha\beta}\, dV_{M}\\
 & =\left\langle T\left|\left(\pi_{\Omega_{0}}^{k}\right)^{\ast}S\right.\right\rangle _{L_{k}^{2}(\Omega)},
\end{align*}
where in the last step we have used $\tilde{\psi}_{\beta}=0$ for
$\beta\notin\Pi_{n}$ and $g^{\alpha\beta}=0$ for $\beta\in\Pi_{n},\alpha\notin\Pi_{n}$.
Moreover, the last computation shows that the embedding $\left(\pi_{\Omega_{0}}^{k}\right)^{\ast}$
is isometric, since 
\[
\pi_{\Omega_{0}}^{k}\left(\pi_{\Omega_{0}}^{k}\right)^{\ast}S=S,
\]
where we have used $\intop_{\mathbb{T}^{s}}\: dV_{\mathbb{T}^{s}}=1$.
The application of the abstract descendant mechanism with $A$ given
by (\ref{eq:mother}) and $B_{0}=B_{1}\coloneqq\bigoplus_{k\in\mathbb{N}}\pi_{\Omega_{0}}^{k}$
provides a way to reduce the dimension of the underlying domain for
an evolutionary problem.  

Applying this reduction process in the particular case $\Omega=\Omega_{0}\times\mathbb{T}^{n-1}\subseteq\mathbb{R}\times\mathbb{T}^{n-1}=M$,
gives a $\left(1+1\right)$-di\-men\-sio\-nal evolutionary descendant
of (\ref{eq:mother-evo},\ref{eq:mother}) on the open subset $\Omega$
of the flat tube manifold $M.$ We may write this descendant in Cartesian
coordinates simply as 

\begin{equation}
\partial_{0}\mathcal{M}\left(\partial_{0}^{-1}\right)+\left(\begin{array}{cc}
0 & \partial_{1}\\
\interior{\partial}_{1} & 0
\end{array}\right),\label{eq:Sturm-Liouville}
\end{equation}
where now $\left(\begin{array}{cc}
0 & \partial_{1}\\
\interior{\partial}_{1} & 0
\end{array}\right)$ is skew-selfadjoint on the space $L^{2}\left(\Omega_{0}\right)\oplus L^{2}\left(\Omega_{0}\right).$ 

If $\Omega_{0}=\mathbb{R}$ we may go one step further, we can decompose
$L^{2}\left(\mathbb{R}\right)$ into orthogonal subspaces 
\[
L^{2}\left(\mathbb{R}\right)=L^{2,\mbox{even}}\left(\mathbb{R}\right)\oplus L^{2,\mbox{odd}}\left(\mathbb{R}\right),
\]
with 
\begin{align*}
L^{2,\mbox{even}}\left(\mathbb{R}\right) & =\left\{ f\in L^{2}\left(\mathbb{R}\right)|f\left(x\right)=f\left(-x\right)\mbox{ for a.e. }x\in\mathbb{R}\right\} ,\\
L^{2,\mbox{odd}}\left(\mathbb{R}\right) & =\left\{ f\in L^{2}\left(\mathbb{R}\right)|f\left(x\right)=-f\left(-x\right)\mbox{ for a.e. }x\in\mathbb{R}\right\} .
\end{align*}
Since $\pi_{L^{2,\mbox{even}}\left(\mathbb{R}\right)}$ and $\pi_{L^{2,\mbox{odd}}\left(\mathbb{R}\right)}$
are compatible with $\partial_{1}=\interior{\partial}_{1}$ we obtain
that 
\[
\partial_{0}\tilde{\mathcal{M}}\left(\partial_{0}^{-1}\right)+\left(\begin{array}{cc}
0 & \partial_{1}\\
\partial_{1} & 0
\end{array}\right)
\]
is the $\left(\pi_{L^{2,\mbox{even}}\left(\mathbb{R}\right)},\pi_{L^{2,\mbox{odd}}\left(\mathbb{R}\right)}\right)$-descendant
of (\ref{eq:Sturm-Liouville}) on $L^{2,\mbox{even}}\left(\mathbb{R}\right)\oplus L^{2,\mbox{odd}}\left(\mathbb{R}\right)$
with 
\[
\tilde{\mathcal{M}}\left(\partial_{0}^{-1}\right)\coloneqq\left(\begin{array}{cc}
\pi_{L^{2,\mbox{even}}\left(\mathbb{R}\right)} & 0\\
0 & \pi_{L^{2,\mbox{odd}}\left(\mathbb{R}\right)}
\end{array}\right)\mathcal{M}\left(\partial_{0}^{-1}\right)\left(\begin{array}{cc}
\pi_{L^{2,\mbox{even}}\left(\mathbb{R}\right)}^{*} & 0\\
0 & \pi_{L^{2,\mbox{odd}}\left(\mathbb{R}\right)}^{*}
\end{array}\right)
\]
as a new material law operator. If $\mathcal{M}\left(\partial_{0}^{-1}\right)$
is block diagonal
\[
\mathcal{M}\left(\partial_{0}^{-1}\right)=\left(\begin{array}{cc}
\mathcal{M}_{00}\left(\partial_{0}^{-1}\right) & 0\\
0 & \mathcal{M}_{11}\left(\partial_{0}^{-1}\right)
\end{array}\right)
\]
then the two rows can be combined into one
\[
\partial_{0}\left(\pi_{L^{2,\mbox{even}}\left(\mathbb{R}\right)}\mathcal{M}_{00}\left(\partial_{0}^{-1}\right)\pi_{L^{2,\mbox{even}}\left(\mathbb{R}\right)}^{*}+\pi_{L^{2,\mbox{odd}}\left(\mathbb{R}\right)}\mathcal{M}_{11}\left(\partial_{0}^{-1}\right)\pi_{L^{2,\mbox{odd}}\left(\mathbb{R}\right)}^{*}\right)+\partial_{1}
\]
on $L^{2}\left(\mathbb{R}\right)=L^{2,\mbox{even}}\left(\mathbb{R}\right)\oplus L^{2,\mbox{odd}}\left(\mathbb{R}\right).$
This is the so-called transport equation in the $\left(1+1\right)-$di\-men\-sio\-nal
case, which thus also is shown to be a descendant of (\ref{eq:mother-evo},\ref{eq:mother})
(for $\Omega=M=\mathbb{R}\times\mathbb{T}^{n-1}$). 
\begin{rem}
(The Transport Equation in $\mathbb{R}^{n}$) Returning to $\mathbb{R}^{n}$
and assuming that by a suitable choice of coordinates the transport
operator $\partial_{0}+a\cdot\partial$ assumes the unitarily equivalent
form 
\[
\partial_{0}+\partial_{1}
\]
with $\partial_{1}$ on a cylinder%
\footnote{Similarly we may consider transport on a period slab, i.e. $\Omega=\mathbb{T}\times\Omega_{0}\subseteq\mathbb{T}\times\mathbb{R}$
as a flat Riemannian manifold.%
} $\Omega\coloneqq\mathbb{R}\times\Omega_{0}\subseteq\mathbb{R}\times\mathbb{R}^{n-1}$.
Here also $\partial_{1}=\interior{\partial}_{1}$. Of course, we could
have more complicated material laws:
\begin{equation}
\partial_{0}\mathcal{M}\left(\partial_{0}^{-1}\right)+\partial_{1}.\label{eq:transport}
\end{equation}
The cross-section $\Omega_{0}$ of the cylinder serves here merely
as a parameter range, since no differentiations in these directions
are involved. Allowing for additional parameter dependence in (\ref{eq:mother-evo},\ref{eq:mother})
would make (\ref{eq:transport}) a descendant of (\ref{eq:mother-evo},\ref{eq:mother}).
\end{rem}

\subsection{Interacting Descendants}

The various descendants of (\ref{eq:mother-evo},\ref{eq:mother})
can interact in many ways to create new models of more complex phenomena.
We shall first discuss a particular interaction based on an alternating
differential forms framework (alternating covariant tensors). Then
we shall turn to the discussion of coupled descendants, where the
coupling only occurs via the material law operator.

\subsubsection{The Extended Maxwell System and the Dirac Equation}

\paragraph{The Extended Maxwell Operator}

Assuming a relatively simple material law of the form
\[
\mathcal{M}\left(\partial_{0}^{-1}\right)=\mathcal{M}_{0}
\]
with $\mathcal{M}_{0}$ continuous, selfadjoint and strictly positive
definite, Maxwell's equations can be reformulated as
\[
\partial_{0}+\sqrt{\mathcal{M}_{0}^{-1}}\left(\begin{array}{cc}
0 & -\left(\interior{d}_{1}\wedge\right)^{*}\\
\interior{d}_{1}\wedge & 0
\end{array}\right)\sqrt{\mathcal{M}_{0}^{-1}}
\]
Here $\interior{d}_{1}\wedge$ is the exterior derivative applied
to covariant 1-tensors (with Dirichlet type boundary condition). By
including alternating tensor fields of all odd orders in the first
block component and of all even orders in the second block component
we arrive at
\[
\partial_{0}+\sqrt{\mathcal{M}_{0}^{-1}}\left(\begin{array}{cc}
0 & -\left(\interior{d}_{1,3}\wedge\right)^{*}\\
\interior{d}_{1,3}\wedge & 0
\end{array}\right)\sqrt{\mathcal{M}_{0}^{-1}}.
\]
Here $\interior{d}_{1,3}\wedge$ is the exterior derivative applied
to the direct sum of alternating tensors of order 1 and 3 (with Dirichlet
type boundary condition). Note that $d\wedge\omega=0$ on 3-forms
in $\mathbb{R}^{3}.$ The material law operator $\mathcal{M}_{0}$
is here of course assumed to be continuous, selfadjoint and strictly
positive definite on the larger space. Adding 
\[
\sqrt{\mathcal{M}_{0}}\left(\begin{array}{cc}
0 & -\interior{d}_{0,2}\wedge\\
\left(\interior{d}_{0,2}\wedge\right)^{*} & 0
\end{array}\right)\sqrt{\mathcal{M}_{0}}
\]
as another descendant%
\footnote{Note that 
\[
\left(\begin{array}{cc}
0 & -\interior{d}_{0,2}\wedge\\
\left(\interior{d}_{0,2}\wedge\right)^{*} & 0
\end{array}\right)=\left(\begin{array}{cc}
0 & -1\\
1 & 0
\end{array}\right)\left(\begin{array}{cc}
0 & -\left(\interior{d}_{0,2}\wedge\right)^{*}\\
\interior{d}_{0,2}\wedge & 0
\end{array}\right)\left(\begin{array}{cc}
0 & 1\\
-1 & 0
\end{array}\right).
\]
} of 
\[
\left(\begin{array}{cc}
0 & -\nabla^{*}\\
\nabla & 0
\end{array}\right),
\]
we obtain a unitarily equivalent variant of the type of operator discussed
in \cite{0579.58030} as the extended Maxwell operator 

\[
\partial_{0}+\sqrt{\mathcal{M}_{0}^{-1}}\left(\begin{array}{cc}
0 & -\left(\interior{d}_{1,3}\wedge\right)^{*}\\
\interior{d}_{1,3}\wedge & 0
\end{array}\right)\sqrt{\mathcal{M}_{0}^{-1}}+\sqrt{\mathcal{M}_{0}}\left(\begin{array}{cc}
0 & -\interior{d}_{0,2}\wedge\\
\left(\interior{d}_{0,2}\wedge\right)^{*} & 0
\end{array}\right)\sqrt{\mathcal{M}_{0}}.
\]
Here $\interior{d}_{0,2}\wedge$ is the exterior derivative applied
to the direct sum of alternating tensors of order 0 and 2 (with Dirichlet
type boundary condition). In Cartesian coordinates this is 
\begin{equation}
\partial_{0}+\sqrt{\mathcal{M}_{0}^{-1}}\left(\begin{array}{cccc}
0 & 0 & 0 & 0\\
0 & 0 & 0 & -\curl\\
0 & 0 & 0 & 0\\
0 & \interior{\curl} & 0 & 0
\end{array}\right)\sqrt{\mathcal{M}_{0}^{-1}}+\sqrt{\mathcal{M}_{0}}\left(\begin{array}{cccc}
0 & 0 & 0 & \interior{\dive}\\
0 & 0 & \interior{\grad} & 0\\
0 & \dive & 0 & 0\\
\grad & 0 & 0 & 0
\end{array}\right)\sqrt{\mathcal{M}_{0}}\label{eq:ext-Max}
\end{equation}
which is a convenient reformulation of Maxwell's equations for regularity
and numerical purposes, compare \cite{TaskinenV07}. That the spatial
part of this extended system is still skew-selfadjoint and that it
can be reduced to the original Maxwell system is due to the fact that
\[
\sqrt{\mathcal{M}_{0}^{-1}}\left(\begin{array}{cccc}
0 & 0 & 0 & 0\\
0 & 0 & 0 & -\curl\\
0 & 0 & 0 & 0\\
0 & \interior{\curl} & 0 & 0
\end{array}\right)\sqrt{\mathcal{M}_{0}^{-1}},\;\sqrt{\mathcal{M}_{0}}\left(\begin{array}{cccc}
0 & 0 & 0 & \interior{\dive}\\
0 & 0 & \interior{\grad} & 0\\
0 & \dive & 0 & 0\\
\grad & 0 & 0 & 0
\end{array}\right)\sqrt{\mathcal{M}_{0}}
\]
are commuting selfadjoint operators, which are indeed annihilating
each other. The possibility of reconstructing the original Maxwell
system assumes a particular form%
\footnote{For these special data it can be shown that components of order 0
and 3 are actually zero. If general right-hand sides are considered
then these components will be non-zero producing what is called ``scalar
waves'' contributions. %
} of the right-hand side, see \cite{0579.58030} for details. By adding
a material law term $\tilde{\mathcal{M}}\left(\partial_{0}^{-1}\right)$
we can allow for more complicated material behavior: 

\[
\partial_{0}+\tilde{\mathcal{M}}\left(\partial_{0}^{-1}\right)+\sqrt{\mathcal{M}_{0}^{-1}}\left(\begin{array}{cccc}
0 & 0 & 0 & 0\\
0 & 0 & 0 & -\curl\\
0 & 0 & 0 & 0\\
0 & \interior{\curl} & 0 & 0
\end{array}\right)\sqrt{\mathcal{M}_{0}^{-1}}+\sqrt{\mathcal{M}_{0}}\left(\begin{array}{cccc}
0 & 0 & 0 & \interior{\dive}\\
0 & 0 & \interior{\grad} & 0\\
0 & \dive & 0 & 0\\
\grad & 0 & 0 & 0
\end{array}\right)\sqrt{\mathcal{M}_{0}}.
\]

~
\begin{rem}
Projecting this further down by eliminating the third row and column
leads to a slightly smaller descendant of the extended Maxwell system
\[
\partial_{0}+\tilde{\tilde{\mathcal{M}}}\left(\partial_{0}^{-1}\right)+\sqrt{\tilde{\mathcal{M}}_{0}^{-1}}\left(\begin{array}{ccc}
0 & 0 & 0\\
0 & 0 & -\curl\\
0 & \interior{\curl} & 0
\end{array}\right)\sqrt{\tilde{\mathcal{M}}_{0}^{-1}}+\sqrt{\tilde{\mathcal{M}}_{0}}\left(\begin{array}{ccc}
0 & 0 & \interior{\dive}\\
0 & 0 & 0\\
\grad & 0 & 0
\end{array}\right)\sqrt{\tilde{\mathcal{M}}_{0}}.
\]
For ``ellipticizing'' Maxwell's equations, e.g. for numerical purposes,
this modification is perfectly sufficient, \cite{Weggler1-2012}.
\end{rem}

\paragraph{The Dirac Operator}

The Dirac operator\index{Dirac operator} $\mathcal{Q}_{0}(\partial_{0},\hat{\partial})$
is usually given as the $(4\times4)-$partial differential expression
with the block matrix form (for mass equal to 1) 
\begin{eqnarray*}
\mathcal{Q}_{0}(\partial_{0},\hat{\partial}) & := & \left(\begin{array}{cc}
\partial_{0}+\i & C(\hat{\partial})\\
C(\hat{\partial}) & \partial_{0}-\i
\end{array}\right).
\end{eqnarray*}
 Here%
\footnote{Note that 
\[
\left(\begin{array}{cc}
\partial_{3} & \partial_{1}-\i\,\partial_{2}\\
\partial_{1}+\i\,\partial_{2} & -\partial_{3}
\end{array}\right)
\]
is an operator quaternion, since it has the form
\[
\left(\begin{array}{cc}
A & -B^{*}\\
B & A^{*}
\end{array}\right),
\]
where $A:D\left(A\right)\subseteq H\to H,\, B:D\left(B\right)\subseteq H\to H$
are closed densely defined linear operators, such that $A$ has a
non-empty resolvent set $\nu\left(A\right)$ and $A,B^{*}$ are commuting,
i.e.
\[
\left(\lambda-A\right)^{-1}B\subseteq B\left(\lambda-A\right)^{-1}
\]
for $\lambda\in\nu\left(A\right)$. If $A,B$ are complex numbers
(as multipliers) this block operator matrix yields a standard representation
of the classical quaternions.%
} $C(\hat{\partial})\;:=\left(\begin{array}{cc}
\partial_{3} & \partial_{1}-\i\,\partial_{2}\\
\partial_{1}+\i\,\partial_{2} & -\partial_{3}
\end{array}\right)=\sum_{k=1}^{3}\,\Pi_{k}\,\partial_{k},$ where 
\[
\Pi_{1}:=\left(\begin{array}{cc}
0 & +1\\
+1 & 0
\end{array}\right),\:\Pi_{2}:=\left(\begin{array}{cc}
0 & -\i\\
+\i & 0
\end{array}\right),\:\Pi_{3}:=\left(\begin{array}{cc}
+1 & 0\\
0 & -1
\end{array}\right)
\]
 are known as Pauli matrices\index{Pauli matrices}. Applying the
unitary transformation given by the block matrix $\frac{1}{\sqrt{2}}\left(\begin{array}{cc}
+\i & \:+1\\
\i & \:-1
\end{array}\right)$ to $\mathcal{Q}_{0}(\partial_{0},\hat{\partial})$ we obtain
\[
\begin{array}{lcl}
\mathcal{Q}_{1}(\partial_{0},\hat{\partial}) & := & \left(\begin{array}{cc}
\partial_{0} & \i-\i\, C(\hat{\partial})\\
\i+\i\, C(\hat{\partial}) & \partial_{0}
\end{array}\right)=\\
 &  & \qquad=\frac{1}{2}\left(\begin{array}{cc}
+\i & +1\\
\i & -1
\end{array}\right)\left(\begin{array}{cc}
-\i\:\partial_{0}+1+C(\hat{\partial})\quad & -\i\:\partial_{0}+1-C(\hat{\partial})\\
-\i\: C(\hat{\partial})+\partial_{0}-\i\quad & -\i\: C(\hat{\partial})-\partial_{0}+\i
\end{array}\right)\\
 &  & \qquad=\frac{1}{2}\left(\begin{array}{cc}
+\i & +1\\
\i & -1
\end{array}\right)\left(\begin{array}{cc}
\partial_{0}+\i\quad & C(\hat{\partial})\\
C(\hat{\partial})\quad & \partial_{0}-\i
\end{array}\right)\left(\begin{array}{cc}
-\i & -\i\\
+1 & -1
\end{array}\right).
\end{array}
\]
The latter may be a preferable form since $\mathcal{Q}_{1}(\partial_{0},\hat{\partial})$
has the typical Hamiltonian form of reversibly evolutionary expressions
of mathematical physics
\[
\left(\begin{array}{cc}
\partial_{0} & -W^{*}\\
W & \partial_{0}
\end{array}\right)\:,
\]
where $W\::=\i+\i\, C(\hat{\partial})=\left(\begin{array}{cc}
\i\partial_{3}+\i & \i\partial_{1}+\partial_{2}\\
\i\partial_{1}-\partial_{2} & -\i\partial_{3}+\i
\end{array}\right)$. 

On first glance the Dirac operator does not seem to fit into the framework
we are discussing here, since it does not appear to be constructed
from descendants of (\ref{eq:mother-evo},\ref{eq:mother}). A closer
inspection, however, shows that the Dirac operator is actually unitarily
equivalent to, i.e. in the above sense a relative of , the extended
Maxwell operator (with a variant of the material law).

To see this connection we separate real and imaginary parts, which
yields that $W$ corresponds to 
\begin{align*}
\left(\begin{array}{cccc}
0 & -1-\partial_{3} & \partial_{2} & -\partial_{1}\\
1+\partial_{3} & 0 & \partial_{1} & \partial_{2}\\
-\partial_{2} & -\partial_{1} & 0 & -1+\partial_{3}\\
\partial_{1} & -\partial_{2} & 1-\partial_{3} & 0
\end{array}\right) & =\left(\begin{array}{cccc}
0 & -1 & 0 & 0\\
1 & 0 & 0 & 0\\
0 & 0 & 0 & -1\\
0 & 0 & 1 & 0
\end{array}\right)+\left(\begin{array}{cccc}
0 & -\partial_{3} & \partial_{2} & -\partial_{1}\\
\partial_{3} & 0 & \partial_{1} & \partial_{2}\\
-\partial_{2} & -\partial_{1} & 0 & \partial_{3}\\
\partial_{1} & -\partial_{2} & -\partial_{3} & 0
\end{array}\right).
\end{align*}

Noting that
\begin{align*}
\left(\begin{array}{cccc}
0 & 1 & 0 & 0\\
0 & 0 & -1 & 0\\
0 & 0 & 0 & -1\\
-1 & 0 & 0 & 0
\end{array}\right)\left(\begin{array}{cccc}
0 & -\partial_{3} & \partial_{2} & -\partial_{1}\\
\partial_{3} & 0 & \partial_{1} & \partial_{2}\\
-\partial_{2} & -\partial_{1} & 0 & \partial_{3}\\
\partial_{1} & -\partial_{2} & -\partial_{3} & \text{0}
\end{array}\right)\left(\begin{array}{cccc}
0 & 0 & 0 & 1\\
1 & 0 & 0 & 0\\
0 & 1 & 0 & 0\\
0 & 0 & 1 & 0
\end{array}\right) & =\left(\begin{array}{cccc}
0 & \partial_{1} & \partial_{2} & \partial_{3}\\
\partial_{1} & 0 & -\partial_{3} & \partial_{2}\\
\partial_{2} & \partial_{3} & 0 & -\partial_{1}\\
\partial_{3} & -\partial_{2} & \partial_{1} & 0
\end{array}\right)\\
 & =\left(\begin{array}{cc}
0 & \dive\\
\grad & \curl
\end{array}\right)
\end{align*}
and 
\begin{align*}
\left(\begin{array}{cccc}
0 & 1 & 0 & 0\\
0 & 0 & -1 & 0\\
0 & 0 & 0 & -1\\
-1 & 0 & 0 & 0
\end{array}\right)\left(\begin{array}{cccc}
0 & -1 & 0 & 0\\
1 & 0 & 0 & 0\\
0 & 0 & 0 & -1\\
0 & 0 & 1 & \text{0}
\end{array}\right)\left(\begin{array}{cccc}
0 & 0 & 0 & 1\\
1 & 0 & 0 & 0\\
0 & 1 & 0 & 0\\
0 & 0 & 1 & 0
\end{array}\right) & =\left(\begin{array}{cccc}
0 & 1 & 0 & 0\\
0 & 0 & -1 & 0\\
0 & 0 & 0 & -1\\
-1 & 0 & 0 & 0
\end{array}\right)\left(\begin{array}{cccc}
-1 & 0 & 0 & 0\\
0 & 0 & 0 & 1\\
0 & 0 & -1 & 0\\
0 & 1 & 0 & \text{0}
\end{array}\right)\\
 & =\left(\begin{array}{cccc}
0 & 0 & 0 & 1\\
0 & 0 & 1 & 0\\
0 & -1 & 0 & 0\\
1 & 0 & 0 & 0
\end{array}\right)
\end{align*}
we obtain the unitary equivalence
\begin{align*}
 & \left(\begin{array}{cc}
\left(\begin{array}{cccc}
0 & 0 & 0 & 1\\
1 & 0 & 0 & 0\\
0 & 1 & 0 & 0\\
0 & 0 & 1 & 0
\end{array}\right) & \left(\begin{array}{cccc}
0 & 0 & 0 & 0\\
0 & 0 & 0 & 0\\
0 & 0 & 0 & 0\\
0 & 0 & 0 & 0
\end{array}\right)\\
\left(\begin{array}{cccc}
0 & 0 & 0 & 0\\
0 & 0 & 0 & 0\\
0 & 0 & 0 & 0\\
0 & 0 & 0 & 0
\end{array}\right) & \left(\begin{array}{cccc}
0 & 1 & 0 & 0\\
0 & 0 & -1 & 0\\
0 & 0 & 0 & -1\\
-1 & 0 & 0 & 0
\end{array}\right)
\end{array}\right)\left(\begin{array}{cc}
0 & -W^{*}\\
W & 0
\end{array}\right)\left(\begin{array}{cc}
\left(\begin{array}{cccc}
0 & 1 & 0 & 0\\
0 & 0 & 1 & 0\\
0 & 0 & 0 & 1\\
1 & 0 & 0 & 0
\end{array}\right) & \left(\begin{array}{cccc}
0 & 0 & 0 & 0\\
0 & 0 & 0 & 0\\
0 & 0 & 0 & 0\\
0 & 0 & 0 & 0
\end{array}\right)\\
\left(\begin{array}{cccc}
0 & 0 & 0 & 0\\
0 & 0 & 0 & 0\\
0 & 0 & 0 & 0\\
0 & 0 & 0 & 0
\end{array}\right) & \left(\begin{array}{cccc}
0 & 0 & 0 & -1\\
1 & 0 & 0 & 0\\
0 & -1 & 0 & 0\\
0 & 0 & -1 & 0
\end{array}\right)
\end{array}\right)=\\
 & =\left(\begin{array}{cc}
\left(\begin{array}{cc}
0 & \left(\begin{array}{ccc}
0 & 0 & 0\end{array}\right)\\
\left(\begin{array}{c}
0\\
0\\
0
\end{array}\right) & \left(\begin{array}{ccc}
0 & 0 & 0\\
0 & 0 & 0\\
0 & 0 & 0
\end{array}\right)
\end{array}\right) & \left(\begin{array}{cc}
0 & \dive\\
\grad & -\curl
\end{array}\right)+\left(\begin{array}{cc}
0 & \left(\begin{array}{ccc}
0 & 0 & -1\end{array}\right)\\
\left(\begin{array}{c}
0\\
0\\
-1
\end{array}\right) & \left(\begin{array}{ccc}
0 & 1 & 0\\
-1 & 0 & 0\\
0 & 0 & 0
\end{array}\right)
\end{array}\right)\\
\left(\begin{array}{cc}
0 & \dive\\
\grad & \curl
\end{array}\right)+\left(\begin{array}{cc}
0 & \left(\begin{array}{ccc}
0 & 0 & 1\end{array}\right)\\
\left(\begin{array}{c}
0\\
0\\
1
\end{array}\right) & \left(\begin{array}{ccc}
0 & 1 & 0\\
-1 & 0 & 0\\
0 & 0 & 0
\end{array}\right)
\end{array}\right) & \left(\begin{array}{cc}
0 & \left(\begin{array}{ccc}
0 & 0 & 0\end{array}\right)\\
\left(\begin{array}{c}
0\\
0\\
0
\end{array}\right) & \left(\begin{array}{ccc}
0 & 0 & 0\\
0 & 0 & 0\\
0 & 0 & 0
\end{array}\right)
\end{array}\right)
\end{array}\right).
\end{align*}

In the free-space situation the Dirac operator $\partial_{0}+\left(\begin{array}{cc}
0 & -W^{*}\\
W & 0
\end{array}\right)$ is thus unitarily equivalent to the extended Maxwell operator 
\[
\partial_{0}+\mathcal{M}_{1}+\left(\begin{array}{cccc}
0 & 0 & 0 & \interior{\dive}\\
0 & 0 & \interior{\grad} & -\curl\\
0 & \dive & 0 & 0\\
\grad & \interior{\curl} & 0 & 0
\end{array}\right)
\]
 where $\mathcal{M}_{1}=\left(\begin{array}{cc}
\left(\begin{array}{cc}
0 & \left(\begin{array}{ccc}
0 & 0 & 0\end{array}\right)\\
\left(\begin{array}{c}
0\\
0\\
0
\end{array}\right) & \left(\begin{array}{ccc}
0 & 0 & 0\\
0 & 0 & 0\\
0 & 0 & 0
\end{array}\right)
\end{array}\right) & \left(\begin{array}{cc}
0 & \left(\begin{array}{ccc}
0 & 0 & -1\end{array}\right)\\
\left(\begin{array}{c}
0\\
0\\
-1
\end{array}\right) & \left(\begin{array}{ccc}
0 & 1 & 0\\
-1 & 0 & 0\\
0 & 0 & 0
\end{array}\right)
\end{array}\right)\\
\left(\begin{array}{cc}
0 & \left(\begin{array}{ccc}
0 & 0 & 1\end{array}\right)\\
\left(\begin{array}{c}
0\\
0\\
1
\end{array}\right) & \left(\begin{array}{ccc}
0 & 1 & 0\\
-1 & 0 & 0\\
0 & 0 & 0
\end{array}\right)
\end{array}\right) & \left(\begin{array}{cc}
0 & \left(\begin{array}{ccc}
0 & 0 & 0\end{array}\right)\\
\left(\begin{array}{c}
0\\
0\\
0
\end{array}\right) & \left(\begin{array}{ccc}
0 & 0 & 0\\
0 & 0 & 0\\
0 & 0 & 0
\end{array}\right)
\end{array}\right)
\end{array}\right)$ is skew-selfadjoint, i.e. from the electrodynamics perspective we
are in a chiral media case.

Thus we have shown that the Dirac equation also fits seamlessly into
our construction of descendants of (\ref{eq:mother-evo},\ref{eq:mother})
and their interaction. In particular, the Dirac operator is a relative%
\footnote{In the framework of quaternions a connection between a differently
extended time-harmonic Maxwell operator and the time-harmonic Dirac
operator has earlier been discovered by Kravchenko and Shapiro, \cite{0305-4470-28-17-030},
compare also \cite{Krav}.%
} of the extended Maxwell operator discussed above. 
\begin{rem}
It is a rather remarkable observation that the Dirac equation is so
closely connected to the extended Maxwell system. It appears from
this perspective that spinors are actually a redundant construction
since the alternating forms setup for the extended Maxwell system
is already quite sufficient to discuss Dirac equations. The interpretation
of this observation is not a mathematical issue but may well be a
matter for theoretical physicists to contemplate.
\end{rem}

\subsubsection{Coupled Systems}

Let us recall from \cite{PDE_DeGruyter} the systematic coupling mechanism
between various different descendants. Without coupling the systems
of interest can be combined simply by writing them together in diagonal
block operator matrix form:
\[
\partial_{0}\left(\begin{array}{c}
V_{0}\\
\vdots\\
\vdots\\
V_{n}
\end{array}\right)+A\left(\begin{array}{c}
U_{0}\\
\vdots\\
\vdots\\
U_{n}
\end{array}\right)=\left(\begin{array}{c}
f_{0}\\
\vdots\\
\vdots\\
f_{n}
\end{array}\right),
\]
where

\begin{align*}
A & =\left(\begin{array}{cccc}
A_{0} & 0 & \cdots & 0\\
0 & \ddots &  & \vdots\\
\vdots &  & \ddots & 0\\
0 & \cdots & 0 & A_{n}
\end{array}\right)
\end{align*}
inherits the skew-selfadjointness in $H=\bigoplus_{k=0,\ldots,n}H_{k}$
from its skew-selfadjoint diagonal entries $A_{k}:D\left(A_{k}\right)\subseteq H_{k}\to H_{k}$,
$k=0,\ldots,n$. The combined material laws here take the simple block
diagonal form
\[
V=\left(\begin{array}{c}
V_{0}\\
\vdots\\
\vdots\\
V_{n}
\end{array}\right)=M^{\mathrm{in}}\left(\partial_{0}^{-1}\right)U\coloneqq\left(\begin{array}{cccc}
M_{00}\left(\partial_{0}^{-1}\right) & 0 & \cdots & 0\\
0 & \ddots &  & \vdots\\
\vdots &  & \ddots & 0\\
0 & \cdots & 0 & M_{nn}\left(\partial_{0}^{-1}\right)
\end{array}\right)\left(\begin{array}{c}
U_{0}\\
\vdots\\
\vdots\\
U_{n}
\end{array}\right).
\]
Coupling between these phenomena now can be modeled by expanding the
material law to contain block off-diagonal entries
\begin{align*}
M^{\mathrm{ex}}\left(\partial_{0}^{-1}\right) & \coloneqq\left(\begin{array}{cccc}
M_{00}\left(\partial_{0}^{-1}\right) & \cdots & \cdots & M_{0n}\left(\partial_{0}^{-1}\right)\\
\vdots & \ddots &  & \vdots\\
\vdots &  & \ddots & \vdots\\
M_{n0}\left(\partial_{0}^{-1}\right) & \cdots & \cdots & M_{nn}\left(\partial_{0}^{-1}\right)
\end{array}\right)-\left(\begin{array}{cccc}
M_{00}\left(\partial_{0}^{-1}\right) & 0 & \cdots & 0\\
0 & \ddots &  & \vdots\\
\vdots &  & \ddots & 0\\
0 & \cdots & 0 & M_{nn}\left(\partial_{0}^{-1}\right)
\end{array}\right).
\end{align*}
The full material law now is of the familiar form 
\[
V=\mathcal{M}\left(\partial_{0}^{-1}\right)U
\]
 with
\begin{align*}
\mathcal{M}\left(\partial_{0}^{-1}\right) & \coloneqq M^{\mathrm{in}}\left(\partial_{0}^{-1}\right)+M^{\mathrm{ex}}\left(\partial_{0}^{-1}\right).
\end{align*}

This coupling mechanism now allows to model thermo-elasticity, thermo-piezo-electro-magnetism
and so on. A number of examples for coupled systems have been discussed
elsewhere, see \cite{Picard2005,Pi2009-1,pre05665378,pre05760017,PIC_2010:1889,McGhee2011}.
The ``philosophy'' of this coupling mechanism is that coupling occurs
only via the material law. 

In the following we shall illustrate the abstract coupling mechanism
with a particular concrete example, which will at the same time serve
to exemplify the construction of descendants of coupled systems.

Starting point of our example collection is the classical system of
thermo-elasticity, which will also allow us to re-iterate the point
made previously in connection with the Dirac operator and the extended
Maxwell operator that indeed systems with very different physical
interpretations may share the same solution theory with differences
being incorporated merely in possibly different material laws (or
their interpretation).

\paragraph{Thermo-Elasticity and Biot's Model for Porous Media}

The classical system of $\left(1+3\right)-$di\-men\-sio\-nal thermo-elasticity%
\footnote{Due to an inconvenient choice of unknowns the original classical system
of thermo-elasticity has an unbounded coupling term, see \cite{Leis:Buch:2}.
The form given here avoids this drawback. More specifically the difference
hinges on the use of stress instead of strain as unknown tensor field.%
} can be described by
\[
\left(\partial_{0}\mathcal{M}\left(\partial_{0}^{-1}\right)+A\right)\left(\begin{array}{c}
\eta\\
\zeta\\
s\\
T
\end{array}\right)=F
\]
with
\[
A\coloneqq\left(\begin{array}{cccc}
0 & -\textrm{div} & 0 & 0\\
-\interior{\textrm{grad}} & 0 & 0 & 0\\
0 & 0 & 0 & -\textrm{Div}\\
0 & 0 & -\interior{\textrm{Grad}} & 0
\end{array}\right).
\]
The classical material law is of the form
\begin{align*}
\mathcal{M}\left(\partial_{0}^{-1}\right) & =\mathcal{M}_{0}+\partial_{0}^{-1}\mathcal{M}_{1}
\end{align*}
with

\[
\mathcal{M}_{0}\coloneqq\left(\begin{array}{cccc}
\nu_{1}+\Gamma^{*}C^{-1}\Gamma & \qquad0 & \qquad0 & \qquad\Gamma^{*}C^{-1}\\
0 & \qquad0 & \qquad0 & \qquad0\\
0 & \qquad0 & \qquad\nu_{2} & \qquad0\\
C^{-1}\Gamma & \qquad0 & \qquad0 & \qquad C^{-1}
\end{array}\right),\;\mathcal{M}_{1}\coloneqq\left(\begin{array}{cccc}
0 & \qquad0 & \qquad0 & \qquad0\\
0 & \qquad\kappa^{-1} & \qquad0 & \qquad0\\
0 & \qquad0 & \qquad0 & \qquad0\\
0 & \qquad0 & \qquad0 & \qquad0
\end{array}\right)
\]

where $\nu_{1},\kappa,\nu_{2},C$ are continuous selfadjoint and strictly
positive definite operators. This system formally coincides with Biot's
porous media model, merely the meaning of the quantities involved,
i.e. the units, have changed, see e.g. \cite{pre05665378}. 

As in all these models we may allow for more complex material laws
as long as (\ref{eq:posdef-2}) is maintained: 
\begin{equation}
\left(\partial_{0}\mathcal{M}\left(\partial_{0}^{-1}\right)+\left(\begin{array}{cccc}
0 & -\textrm{div} & 0 & 0\\
-\interior{\textrm{grad}} & 0 & 0 & 0\\
0 & 0 & 0 & -\textrm{Div}\\
0 & 0 & -\interior{\textrm{Grad}} & 0
\end{array}\right)\right)\left(\begin{array}{c}
\eta\\
\zeta\\
s\\
T
\end{array}\right)=F.\label{eq:biot-thermo-elastic}
\end{equation}

\paragraph{Reissner-Mindlin Plate}

Assuming $\Omega\coloneqq\Omega_{0}\times\mathbb{T}\subseteq\mathbb{R}^{2}\times\mathbb{T}\eqqcolon M$
(instead of $M=\mathbb{R}^{3}$) we can reduce (\ref{eq:biot-thermo-elastic})
by one spatial dimension to a (1+2)-di\-men\-sio\-nal evolutionary
problem following the strategy in Section \ref{sub:Reducing-Dimensions}.
Indeed, the resulting evolutionary equation looks the same, but now
it has to be interpreted in $L_{0}^{2}\left(\Omega_{0}\right)\oplus L_{1}^{2}\left(\Omega_{0}\right)\oplus L_{1}^{2}\left(\Omega_{0}\right)\oplus\mathrm{sym}\left[L_{2}^{2}\left(\Omega_{0}\right)\right]$
with $\Omega_{0}\subseteq\mathbb{R}^{2}.$ With
\[
F\eqqcolon\left(\begin{array}{c}
f\\
0\\
g\\
0
\end{array}\right)
\]
and

\begin{align*}
\mathcal{M}\left(\partial_{0}^{-1}\right) & =\mathcal{M}_{0}+\partial_{0}^{-1}\mathcal{M}_{1}
\end{align*}
with

\[
\mathcal{M}_{0}\coloneqq\left(\begin{array}{cccc}
\nu_{1} & \qquad0 & \qquad0 & \qquad0\\
0 & \qquad\kappa & \qquad0 & \qquad0\\
0 & \qquad0 & \qquad\nu_{2} & \qquad0\\
0 & \qquad0 & \qquad0 & \qquad C^{-1}
\end{array}\right),\;\mathcal{M}_{1}\coloneqq\left(\begin{array}{cccc}
d & \qquad0 & \qquad0 & \qquad0\\
0 & \qquad0 & \qquad-1 & \qquad0\\
0 & \qquad1 & \qquad0 & \qquad0\\
0 & \qquad0 & \qquad0 & \qquad0
\end{array}\right)
\]

where $\nu_{1},\kappa,\nu_{2},C$ continuous selfadjoint and strictly
positive definite (with physically different meaning, i.e. different
units!) we obtain the Reissner-Mindlin plate model. Coupling occurs
here via $\mathcal{M}_{1}.$

Note that by reducing this to a second order system (by substituting
the equations from rows 2 and 4 into the remaining two equations)
we obtain the perhaps more familiar form of the Reissner-Mindlin model
(with homogeneous Dirichlet boundary condition)
\begin{equation}
\begin{array}{rl}
\nu_{1}\partial_{0}^{2}\tilde{\eta}-\dive\kappa^{-1}\left(\interior{\textrm{grad}}\tilde{\eta}+\tilde{s}\right)+d\partial_{0}\tilde{\eta} & =f,\\
\nu_{2}\partial_{0}^{2}\tilde{s}-\Div C\interior{\Grad}\tilde{s}+\kappa^{-1}\left(\interior{\textrm{grad}}\tilde{\eta}+\tilde{s}\right) & =g,
\end{array}\label{eq:RM-plate}
\end{equation}
where $\tilde{\eta}\coloneqq\partial_{0}^{-1}\eta$, $\tilde{s}\coloneqq\partial_{0}^{-1}s.$

For the damping coefficient $d=0$ we have that the system is conservative,
since 
\[
\sqrt{\mathcal{M}_{0}^{-1}}\left(\left(\begin{array}{cccc}
0 & 0 & 0 & 0\\
0 & 0 & -1 & 0\\
0 & 1 & 0 & 0\\
0 & 0 & 0 & 0
\end{array}\right)+\left(\begin{array}{cccc}
0 & -\textrm{div} & 0 & 0\\
-\interior{\textrm{grad}} & 0 & 0 & 0\\
0 & 0 & 0 & -\textrm{Div}\\
0 & 0 & -\interior{\textrm{Grad}} & 0
\end{array}\right)\right)\sqrt{\mathcal{M}_{0}^{-1}}
\]
is skew-selfadjoint and thus generates a unitary group leading to
norm conservation for pure initial value problems.  
\begin{rem}
~ \end{rem}
\begin{enumerate}
\item (A note on the Kirchhoff-Love plate)

Letting in
\[
\partial_{0}\left(\begin{array}{cccc}
\nu_{1} & \qquad0 & \qquad0 & \qquad0\\
0 & \qquad\kappa & \qquad0 & \qquad0\\
0 & \qquad0 & \qquad\nu_{2} & \qquad0\\
0 & \qquad0 & \qquad0 & \qquad C^{-1}
\end{array}\right)+\left(\begin{array}{cccc}
d & \qquad0 & \qquad0 & \qquad0\\
0 & \qquad0 & \qquad-1 & \qquad0\\
0 & \qquad1 & \qquad0 & \qquad0\\
0 & \qquad0 & \qquad0 & \qquad0
\end{array}\right)+A
\]
$\kappa=0$ and $\nu_{2}=0$ (in consequence destroying well-posedness
for associated initial boundary value problems) yields
\[
\partial_{0}\left(\begin{array}{cccc}
\nu_{1} & \qquad0 & \qquad0 & \qquad0\\
0 & \qquad0 & \qquad0 & \qquad0\\
0 & \qquad0 & \qquad0 & \qquad0\\
0 & \qquad0 & \qquad0 & \qquad C^{-1}
\end{array}\right)+\left(\begin{array}{cccc}
d & \qquad0 & \qquad0 & \qquad0\\
0 & \qquad0 & \qquad-1 & \qquad0\\
0 & \qquad1 & \qquad0 & \qquad0\\
0 & \qquad0 & \qquad0 & \qquad0
\end{array}\right)+A
\]
Eliminating the second and third unknowns and equations yields 
\begin{align}
\left(\partial_{0}\left(\begin{array}{cc}
\nu_{1} & 0\\
0 & C^{-1}
\end{array}\right)+\left(\begin{array}{cc}
d & 0\\
0 & 0
\end{array}\right)+\left(\begin{array}{cc}
0 & -\dive\Div\\
\Grad\grad & 0
\end{array}\right)\right)\left(\begin{array}{c}
\eta\\
T
\end{array}\right) & =\left(\begin{array}{c}
f\\
0
\end{array}\right).\label{eq:KL}
\end{align}
This is the Kirchhoff-Love plate model, which by a suitable choice
of boundary condition is again accessible to the abstract solution
theory of evolutionary equations, see e.g. \cite{PDE_DeGruyter}.

\item Following the ``logic'' of the transition from the Reissner-Mindlin
plate to the Kirchoff-Love plate we could also formally obtain the
real Schr{\"o}dinger operator $\partial_{0}+\left(\begin{array}{cc}
0 & -\Delta_{D}\\
\Delta_{D} & 0
\end{array}\right),$ see Section \ref{sub:Acoustic-Equation--Heat}, from the first order
system 
\[
\partial_{0}\left(\begin{array}{cccc}
1 & 0 & 0 & 0\\
0 & \epsilon & 0 & 0\\
0 & 0 & \epsilon & 0\\
0 & 0 & 0 & 1
\end{array}\right)+\left(\begin{array}{cccc}
0 & 0 & 0 & 0\\
0 & 0 & -1 & 0\\
0 & 1 & 0 & 0\\
0 & 0 & 0 & 0
\end{array}\right)+\left(\begin{array}{cccc}
0 & -\sqrt{-\Delta_{D}} & 0 & 0\\
\sqrt{-\Delta_{D}} & 0 & 0 & 0\\
0 & 0 & 0 & \sqrt{-\Delta_{D}}\\
0 & 0 & -\sqrt{-\Delta_{D}} & 0
\end{array}\right)
\]
by similarly letting $\epsilon=0$ and eliminating the second and
third components and equations. 
\end{enumerate}

\paragraph{The Timoshenko Beam }

Assuming $\Omega_{0}\coloneqq\Omega_{1}\times\mathbb{T}\subseteq\mathbb{R}\times\mathbb{T}\eqqcolon M$
(instead of $\Omega_{0}\subseteq\mathbb{R}^{2}$) for the Reissner-Mindlin
plate, following the arguments in Section \ref{sub:Reducing-Dimensions},
we can reduce this model further to the $\left(1+1\right)$-di\-men\-sional
case, which leads to the Timoshenko beam model. In Cartesian coordinates
this is now 
\[
\left(\partial_{0}\mathcal{M}\left(\partial_{0}^{-1}\right)+\left(\begin{array}{cccc}
0 & -\partial_{1} & 0 & 0\\
-\interior{\partial}_{1} & 0 & 0 & 0\\
0 & 0 & 0 & -\partial_{1}\\
0 & 0 & -\interior{\partial}_{1} & 0
\end{array}\right)\right)\left(\begin{array}{c}
\eta\\
\zeta\\
s\\
T
\end{array}\right)=\left(\begin{array}{c}
f\\
0\\
g\\
0
\end{array}\right),
\]
where the material law has the same shape as before
\begin{align*}
\mathcal{M}\left(\partial_{0}^{-1}\right) & =\mathcal{M}_{0}+\partial_{0}^{-1}\mathcal{M}_{1}
\end{align*}
with

\[
\mathcal{M}_{0}\coloneqq\left(\begin{array}{cccc}
\nu_{1} & \qquad0 & \qquad0 & \qquad0\\
0 & \qquad\kappa & \qquad0 & \qquad0\\
0 & \qquad0 & \qquad\nu_{2} & \qquad0\\
0 & \qquad0 & \qquad0 & \qquad C^{-1}
\end{array}\right),\;\mathcal{M}_{1}\coloneqq\left(\begin{array}{cccc}
d & \qquad0 & \qquad0 & \qquad0\\
0 & \qquad0 & \qquad-1 & \qquad0\\
0 & \qquad1 & \qquad0 & \qquad0\\
0 & \qquad0 & \qquad0 & \qquad0
\end{array}\right)
\]

where $\nu_{1},\kappa,\nu_{2},C$ continuous selfadjoint and strictly
positive definite. Of course physical meaning (the units) have changed
once again. Here also the second order system may be more familiar
\begin{equation}
\begin{array}{rl}
\nu_{1}\partial_{0}^{2}\tilde{\eta}-\partial_{1}\kappa^{-1}\left(\interior{\partial}_{1}\tilde{\eta}+\tilde{s}\right)+d\partial_{0}\tilde{\eta} & =f,\\
\nu_{2}\partial_{0}^{2}\tilde{s}-\partial_{1}C\interior{\partial}_{1}\tilde{s}+\kappa^{-1}\left(\interior{\partial}_{1}\tilde{\eta}+\tilde{s}\right) & =g,
\end{array}\label{eq:Timoshenko}
\end{equation}
where $\tilde{\eta}\coloneqq\partial_{0}^{-1}\eta$, $\tilde{s}\coloneqq\partial_{0}^{-1}s,$
compare (\ref{eq:RM-plate}).
\begin{rem}
(Euler-Bernoulli Beam) Repeating the questionable ``construction''
of the Kirchhoff-Love plate model for the Timoshenko beam, leads to
the Euler-Bernoulli beam model. 
\end{rem}
With this we conclude our tour through various examples underscoring
the deep connectedness of seemingly very different mathematical models.
We have seen, how various particular dynamic linear model equations
can be extracted from the mother operator (\ref{eq:mother-evo},\ref{eq:mother})
assuming different material laws. The solution theory itself rests
simply on strict positive definiteness. One may well wonder if the
simplicity and transparency of these structural observations could
not give rise to a ``grand unified'' numerical scheme. %

\end{document}